\documentclass[reqno]{amsart}
\usepackage{bm,amsmath,amsthm,amssymb,mathtools,verbatim,amsfonts,tikz-cd,mathrsfs}
\usepackage{enumitem}
\usepackage{float}

\usepackage[hidelinks,colorlinks=true,linkcolor=blue, citecolor=black,linktocpage=true]{hyperref}
\usepackage{graphicx}
\usepackage[alphabetic]{amsrefs}
\usepackage{caption}
\usepackage{morefloats}
\usepackage[top=1.3in, bottom=1.1in, left=1.3in, right=1.3in,marginpar=1in]{geometry}
\usepackage{subfigure}

\usetikzlibrary{matrix,arrows,decorations.pathmorphing}

\usepackage{chngcntr}
\counterwithin{figure}{section}

\newtheorem{thm}{Theorem}[section]
\newtheorem{prop}[thm]{Proposition}

\newtheorem{lem}[thm]{Lemma}

\theoremstyle{definition}
\newtheorem{define}[thm]{Definition}

\theoremstyle{remark}
\newtheorem{rem}[thm]{Remark}
\newtheorem{example}[thm]{Example}

\newcommand{\ve}[1]{\boldsymbol{\mathbf{#1}}}

\newcommand{\Z}{\mathbb{Z}}

\newcommand{\Q}{\mathbb{Q}}

\renewcommand{\d}{\partial}
\renewcommand{\subset}{\subseteq}

\renewcommand{\bar}{\overline}

\newcommand{\iso}{\cong}

\DeclareMathOperator{\gr}{{gr}}

\DeclareMathOperator{\id}{{id}}

\DeclareMathOperator{\im}{{im}}

\DeclareMathOperator{\rank}{{rank}}

\DeclareMathOperator{\spin}{{spin}}

\DeclareMathOperator{\sgn}{{sgn}}

\newcommand{\bF}{\mathbb{F}}

\newcommand{\cA}{\mathcal{A}}
\newcommand{\cB}{\mathcal{B}}
\newcommand{\cC}{\mathcal{C}}
\newcommand{\cD}{\mathcal{D}}

\newcommand{\cS}{\mathcal{S}}

\newcommand{\cW}{\mathcal{W}}
\newcommand{\cX}{\mathcal{X}}
\newcommand{\cY}{\mathcal{Y}}
\newcommand{\cZ}{\mathcal{Z}}

\newcommand{\frI}{\mathfrak{I}}

\newcommand{\frs}{\mathfrak{s}}

\newcommand{\cCFK}{\mathcal{C\hspace{-.5mm}F\hspace{-.3mm}K}}

\newcommand{\CF}{\mathit{CF}}

\newcommand{\xs}{\ve{x}}

\newcommand{\zs}{\ve{z}}
\newcommand{\ws}{\ve{w}}

\renewcommand{\a}{\alpha}
\renewcommand{\b}{\beta}
\newcommand{\g}{\gamma}

\usepackage{leftidx}

\newcommand{\Ss}[1]{\scriptstyle{#1}}

\numberwithin{equation}{section}

\newcommand{\ar}{\mathrm{a.r.}}

\newcommand\co{\colon}
\newcommand{\lab}[1]{$\scriptstyle #1$}

\newcommand{\knotU}{\mathscr{U}}
\newcommand{\knotV}{\mathscr{V}}

\newcommand{\scU}{\mathscr{U}}
\newcommand{\scV}{\mathscr{V}}

\newcommand{\SF}{\mathit{SF}}

\allowdisplaybreaks

\title{On the quotient of the homology cobordism group by Seifert spaces}
\author[K. Hendricks]{Kristen Hendricks}
\thanks{KH was partially supported by NSF grant DMS-2019396 and a Sloan Research Fellowship.}
\address{Department of Mathematics, Rutgers University, New Brunswick, NJ, USA}
\email{kristen.hendricks@rutgers.edu}
\author[J. Hom]{Jennifer Hom}
\thanks{JH was partially supported by NSF grant DMS-1552285.}
\address{School of Mathematics, Georgia Institute of Technology, Atlanta, GA, USA}
\email{hom@math.gatech.edu}
\author[M. Stoffregen]{Matthew Stoffregen}
\address{Department of Mathematics, Michigan State University, East Lansing, MI, USA}
\email{stoffre1@msu.edu}
\thanks{MS was partially supported by NSF grant DMS-1952762.}
\author[I. Zemke]{Ian Zemke}
\address{Department of Mathematics\\Princeton University\\  Princeton, NJ, USA}
\email{izemke@math.princeton.edu}
\thanks{IZ was partially supported by NSF grant DMS-1703685.}


\begin{document}

\begin{abstract} 
We prove that the quotient of the integer homology cobordism group by the subgroup generated by the Seifert fibered spaces is infinitely generated.
\end{abstract}

\maketitle


\section{Introduction}
The homology cobordism group $\Theta^3_\Z$ consists of integer homology 3-spheres modulo integer homology cobordism and is a fundamental structure in geometric topology. For example,  $\Theta^3_\Z$ played a central role in Manolescu's \cite{ManolescuPin2Triangulation} disproof of the triangulation conjecture in high dimensions. 

A natural question to ask is which types of manifolds can represent a given class $[Y] \in \Theta^3_\Z$. The first answers to this question were in the positive direction. Livingston \cite{livingston} showed that every class in $\Theta^3_\Z$ can be represented by an irreducible integer homology sphere, and Myers \cite{myers} improved this to show that every class admits a hyperbolic representative. More recently, Mukherjee \cite[Theorem 1.18]{Mukherjee} showed that every class admits a Stein fillable representative.

In the negative direction, Fr{\o}yshov (in unpublished work), F. Lin \cite{LinExact}, and Stoffregen \cite{stoffregen-sums} showed that there are classes in $\Theta^3_\Z$ that do not admit a Seifert fibered representative. Nozaki, Sato, and Taniguchi \cite[Corollaries 1.6 and 1.7]{NST} improved this result to show that there are classes that admit neither a Seifert fibered representative nor a representative that is surgery on a knot in $S^3$. The Fr{\o}yshov, Stoffregen, and Nozaki-Sato-Taniguchi examples are all connected sums of Seifert fibered spaces, and Lin's example has Floer homology consistent with it being representable by a Seifert fibered space. In particular, these results are insufficient to show $\Theta^3_\Z$ is not generated by Seifert fibered spaces.

Using the involutive Heegaard Floer homology of Hendricks and Manolescu \cite{HMInvolutive}, we proved in \cite[Theorem 1.9]{HHSZ} that Seifert fibered spaces do not generate $\Theta^3_\Z$. More precisely, let $\Theta_{\SF}$ denote the subgroup of $\Theta^3_\Z$ generated by Seifert fibered spaces. We showed that the quotient $\Theta^3_\Z / \Theta_{\SF}$ contains a subgroup isomorphic to $\Z$, generated by $Y = S^3_{+1}(-2T_{6,7} \# T_{6,13} \# T_{-2,3; 2,5})$. The main result of this paper is that the quotient $\Theta^3_\Z / \Theta_{\SF}$ is in fact infinitely generated:

\begin{thm}\label{thm:main}
The quotient $\Theta^3_\Z / \Theta_{\SF}$ contains a subgroup isomorphic to $\Z^\infty$, spanned by 
\[ Y_n = S^3_{+1}(T_{2,3} \# -2T_{2n, 2n+1} \# T_{2n, 4n+1}), \quad n \geq 3, \text{ } n\text{ odd}. \]
\end{thm}

Involutive Heegaard Floer homology associates to an integer homology sphere $Y$ (or more generally a spin rational homology sphere) an algebraic object called an iota-complex. The local equivalence class of this iota-complex is an invariant of the homology cobordism class of $Y$, and the set of iota-complexes modulo local equivalence forms a group under tensor product. For technical reasons, it is often convenient to consider a slightly weaker notion of equivalence, called almost local equivalence, and the associated group $\widehat{\frI}$ of almost iota-complexes modulo almost local-equivalence, as in \cite{DHSThomcob}. There is a group homomorphism
\[ \hat{h} \co \Theta^3_\Z \to \widehat{\frI} \]
induced by sending $[Y]$ to the almost local equivalence class of its iota-complex. 

The proof of Theorem \ref{thm:main} relies on the following steps:
\begin{enumerate}
	\item A computation of the almost local equivalence class of the iota-complex associated to $Y_n$ using the involutive surgery formula of \cite[Theorem 1.6]{HHSZ}. We call this complex $C(n-1)$.
	\item \label{it:lincombo} A computation of the almost local equivalence class of linear combinations of $C(n-1)$, for $n\ge 2$, following the strategy of \cite[Section 8.1]{DHSThomcob}.
	\item A comparison of the results from step~\eqref{it:lincombo} with the computation of $\hat{h}(\Theta_\SF)$ in \cite[Theorem 8.1]{DHSThomcob}.
\end{enumerate}

\begin{rem}
Let $\Theta_{\mathit{AR}}$ denote the subgroup of $\Theta^3_\Z$ spanned by almost-rationally plumbed 3-manifolds; see \cite{NemethiAR} for the precise definition of an almost-rational plumbing. By \cite[Theorem 1.1]{dai-stoffregen}, $\hat{h}(\Theta_{\mathit{AR}}) = \hat{h}(\Theta_\SF)$, so the proof of Theorem \ref{thm:main} actually shows that the quotient $\Theta^3_\Z/ \Theta_{\mathit{AR}}$ contains subgroup isomorphic to $\Z^\infty$.
\end{rem}

Recall that a \emph{graph manifold} is a prime 3-manifold whose JSJ decomposition contains only Seifert fibered pieces. The manifolds $Y_n$ in Theorem \ref{thm:main} are all graph manifolds, since they are surgery along connected sums of torus knots.  Similarly, the manifold $Y$ in \cite[Theorem 1.9]{HHSZ} is a graph manifold, since it is surgery along a connected sum of iterated torus knots. A natural question to ask is whether every homology sphere is homology cobordant to a graph manifold, or more generally, whether graph manifolds generate $\Theta^3_\Z$. As far as the authors know, both of these questions remain open; we expect that the answer to both is no. Note that if \cite[Conjecture 1.22]{NST} is true, then graph manifolds do not generate $\Theta^3_\Z$, as pointed out in \cite[Proposition 1.23]{NST}. Another natural question to ask is whether surgeries on knots in $S^3$ generate $\Theta^3_\Z$.

\subsection*{Organization} This paper is organized as follows. In Section~\ref{section-2-background} we recall some background on involutive Heegaard Floer homology. In Section~\ref{section-3-iotacomplex} we prove that the almost iota-complex of the manifolds $Y_n$ in Theorem \ref{thm:main} is $C(n-1)$. In Section~\ref{section-4-tensorproducts} we compute the almost local equivalence classes of linear combinations of $C(n)$, and use it to complete the proof of Theorem \ref{thm:main}.

\subsection*{Acknowledgements} We are grateful to Irving Dai and Linh Truong for helpful conversations. The search for the examples in this paper was inspired by the examples of knot-like complexes computed by Wenzhao Chen in \cite{ChenMazur} (although our ultimate examples and methods are significantly different from his). We also thank the anonymous referees for helpful comments and feedback.

\section{Background on involutive Heegaard Floer homology} \label{section-2-background}

 We will assume the reader is familiar with the basics of knot Floer homology \cite{OSKnots} \cite{RasmussenKnots}, and confine ourselves to listing some definitions necessary for studying involutive Heegaard Floer homology \cite{HMInvolutive}.  In fact, in the present paper we will only need a few properties of this theory, which we summarize here.  For more details, see \cite[Section 3]{HHSZ}.

 \begin{define}
 	\label{def:iota-complex}
 	An \emph{iota-complex} (or \emph{$\iota$-complex}) $(C,\iota)$ is a chain complex $C$, which is free and finitely generated over $\bF[U]$, equipped with an endomorphism $\iota$.  Here $\bF$ is the field of $2$ elements, and $U$ is a formal variable with grading $-2$.  Furthermore, the following hold:
 	\begin{enumerate}
 		\item\label{iota-1} $C$ is equipped with a $\Z$-grading, compatible with the action of $U$.  We call this grading the \emph{Maslov} or \emph{homological} grading.
 		\item\label{iota-2} There is a grading-preserving isomorphism $U^{-1}H_*(C)\iso \bF[U,U^{-1}]$.
 		\item \label{iota-3} $\iota$ is a grading-preserving chain map and $\iota^2\simeq \id$.
 	\end{enumerate}
 \end{define}
 
Given two iota-complexes $(C_1, \iota_1)$ and $(C_2, \iota_2)$, a homogeneously graded $\bF[U]$-chain map $f \colon C_1 \to C_2$ is said to be an \emph{$\iota$-homomorphism} if $\iota_2\circ  f+f\circ  \iota_1\simeq 0$. Two iota-complexes $(C_1,\iota_1)$ and $(C_2,\iota_2)$ are called \emph{$\iota$-equivalent} if there is a homotopy equivalence $\Phi\colon C_1\to C_2$ which is an $\iota$-homomorphism. 

For any closed oriented $3$-manifold $Y$ equipped with self-conjugate $\spin^c$ structure $\frs$, Hendricks--Manolescu \cite{HMInvolutive} prove that the $\bF[U]$-chain complex with homotopy involution $(\CF^-(Y,\frs),\iota)$ is well defined up to homotopy-equivalence. 
In the case that $Y$ is a rational homology $3$-sphere, $(\CF^-(Y,\frs),\iota)$ is an iota-complex.

 
 The tensor product of iota-complexes $(C_1,\iota_1)$ and $(C_2,\iota_2)$ is given by 
 \begin{equation}
 (C_1,\iota_1)\otimes(C_2,\iota_2):=(C_1\otimes_{\bF[U]}C_2,\iota_1\otimes \iota_2).
\label{eq:group-relation-I}
 \end{equation}
  Moreover, Hendricks--Manolescu--Zemke \cite{HMZConnectedSum} establish that 
 \[
 (\CF^-(Y_1\# Y_2,\frs_1\# \frs_2),\iota)\simeq (\CF^-(Y_1,\frs_1),\iota_1)\otimes (\CF^-(Y_2,\frs_2),\iota_2),
 \]
 where $\simeq$ denotes homotopy-equivalence of iota-complexes.

 \begin{define} Suppose $(C,\iota)$ and $(C',\iota')$ are two iota-complexes.
 	\begin{enumerate}
 		\item A \emph{local map} from $(C,\iota)$ to $(C',\iota')$ is a grading-preserving $\iota$-homomorphism $F\colon C\to C'$, which induces an isomorphism from $U^{-1} H_*(C)$ to $U^{-1} H_*(C')$.
 		\item We say that $(C,\iota)$ are $(C',\iota')$ are \emph{locally equivalent} if there is a local map from $(C,\iota)$ to $(C',\iota')$, as well as a local map from $(C',\iota')$ to $(C,\iota)$.
 	\end{enumerate}
 \end{define}
 The set of local equivalence classes forms an abelian group, denoted $\frI$, with product given by the operation $\otimes$ in equation~\eqref{eq:group-relation-I}.
  See \cite{HMZConnectedSum}*{Section~8}. Inverses are given by dualizing both the chain complex $C$ and the map $\iota$ with respect to $\bF[U]$; we write $-(C,\iota)$ for this dual iota-complex. According to \cite{HMZConnectedSum}*{Theorem~1.8}, the map
 \[
 Y\mapsto [(\CF^-(Y), \iota)]
 \]
 determines a homomorphism from $\Theta_{\Z}^3$ to $\frI$.

 There is an additional, weaker, equivalence relation between iota-complexes, introduced in \cite{DHSThomcob} (see also \cite[Section 3.3]{HHSZ}).
 
 \begin{define}[\cite{DHSThomcob}*{Definition 3.1}]
 	Let $C_1$ and $C_2$ be free, finitely generated chain complexes over $\bF[U]$,  such that each $C_i$ has an absolute $\Q$-grading and a relative $\Z$-grading with respect to which $U$ has grading $-2$. Two grading-preserving $\bF[U]$-module homomorphisms
 	\[ 
 	f, g 
 	\colon
 	C_1 \rightarrow C_2
 	\]
 	are \emph{homotopic mod $U$}, denoted $f \simeq g \mod U$, if there exists an $\bF[U]$-module homomorphism $H \colon C_1 \rightarrow C_2$ such that $H$ increases grading by one and
 	\[
 	f + g + H \circ \d + \d \circ H \in \im U.
 	\]
 \end{define}

  \begin{define}[\cite{DHSThomcob}*{Definition 3.2}]
 An \emph{almost iota-complex} (or \emph{almost $\iota$-complex}) $\cC=(C,\bar{\iota})$ consists of the following data:
 \begin{itemize}
 \item A free, finitely-generated, $\Z$-graded chain complex $C$ over $\bF[U]$, with
 \[
U^{-1} H_*(C)\iso \bF[U,U^{-1}]. 
 \]
 Here $U$ has degree $-2$ and $U^{-1} H_*(C)$ is supported in even gradings.
 \item A grading-preserving $\bF[U]$-module homomorphism $\bar{\iota}\colon C\to C$ such that
 \[
\bar{\iota}\circ \d+\d \circ \bar{\iota}\in \im U\qquad \text{and} \qquad \bar{\iota}^2\simeq \id \mod U. 
 \]
 \end{itemize}
 \end{define}
 
Of course, any iota-complex induces an almost iota-complex. The definition of tensor product of almost iota-complexes is the same as equation~\eqref{eq:group-relation-I}.
 
In analogy with the terminology above, an \emph{almost $\iota$-homomorphism} from $(C_1, \bar{\iota}_1)$ to $(C_2, \bar{\iota}_2)$ is a homogeneously-graded, $\bF[U]$-equivariant chain map $f \colon C_1 \rightarrow C_2$ such that $f \circ \bar{\iota} \simeq \bar{\iota} \circ f \mod U.$ We then have the following new relation between almost $\iota$-complexes.

\begin{define}[\cite{DHSThomcob}*{Definition~3.5}]Suppose $(C_1,\bar{\iota}_1)$ and $(C_2,\bar{\iota}_2)$ are almost $\iota$-complexes.
\begin{enumerate}
\item An \emph{almost local map} from $(C_1,\bar{\iota}_1)$ to $(C_2,\bar{\iota}_2)$ is a grading-preserving almost $\iota$-homomorphism $F\colon C_1\to C_2$, which induces an isomorphism from $U^{-1} H_*(C)$ to $U^{-1} H_*(C')$.
\item We say that $(C_1,\bar{\iota}_1)$ are $(C_2,\bar{\iota}_2)$ are \emph{almost locally equivalent} if there is an almost local map from $(C_1,\bar{\iota}_1)$ to $(C_2,\bar{\iota}_2)$, as well as an almost local map from $(C_2,\bar{\iota}_2)$ to $(C_1,\bar{\iota}_1)$.
\end{enumerate}
\end{define}

One special case of this definition will be especially useful to us: if $\bar{\iota}$ and $\bar{\iota}'$ are maps on the same complex $C$ such that $(C, \bar{\iota})$ and $(C, \bar{\iota}')$ are each almost iota-complexes, and the difference $\bar{\iota} - \bar{\iota}' \in \im(U)$, then the identity map from $C$ to itself is an almost local equivalence between $(C, \bar{\iota})$ and $(C, \bar{\iota}')$.
 
Using the definitions above, one may construct an almost local equivalence group $\widehat{\frI}$ of almost iota-complexes. It is a non-trivial result that $\widehat{\frI}$ can be parametrized explicitly \cite{DHSThomcob}*{Theorem~6.2}, as we now describe.  To a sequence $(a_1, b_2, a_3, b_4, \dots, a_{2m-1}, b_{2m})$, where $a_i \in \{ \pm\}$ and $b_i \in \Z \setminus \{0\}$, we may associate an almost iota-complex 
 \[
 C(a_1, b_2, a_3, b_4, \dots, a_{2m-1}, b_{2m}),
 \]
  called the \emph{standard complex of type $(a_1, b_2, a_3, b_4, \dots, a_{2m-1}, b_{2m})$}, as follows.  The standard complex is freely generated over $\bF[U]$ by $t_0, t_1, \dots, t_{2m}$. For each symbol $a_i$, we introduce an $\omega=(1+\iota)$-arrow between $t_{i-1}$ and $t_i$ as follows:
 \begin{itemize}
 	\item If $a_i = +$, then $\omega t_{i} = t_{i-1}$.
 	\item If $a_i = -$, then $\omega t_{i-1} = t_i$. 
 \end{itemize}
 For each symbol $b_i$, 
 we introduce a $\d$-arrow between $t_{i-1}$ and $t_i$ as follows:
 \begin{itemize}
 	\item If $b_i>0$, then $\d t_{i} = U^{|b_i|} t_{i-1}$.
 	\item If $b_i<0$, then $\d t_{i-1} = U^{|b_i|} t_{i}$.
 \end{itemize}
 
In computations with standard complexes, it will frequently be convenient to represent the group operation with $+$ instead of $\otimes$. The dual of the standard complex $C(a_1, b_2, a_3, b_4, \dots, a_{2m-1}, b_{2m})$ is the standard complex $-C(a_1, b_2, a_3, b_4, \dots, a_{2m-1}, b_{2m}) = C(-a_1, -b_2, -a_3, -b_4, \dots, -a_{2m-1}, -b_{2m})$, where if $a_i$ is $+$ then $-a_i$ is $-$ and vice versa.

Every element of $\widehat{\frI}$ is locally equivalent to a unique standard complex \cite{DHSThomcob}*{Theorem~6.2}.  Thus, in spite of $\widehat{\frI}$ being infinitely-generated, its elements are easy to describe. We write $\hat{h}$ for the composite
\[
\Theta_{\Z}^3\to \frI\to \widehat{\frI}.
\]

Note that there is not a simple formula for the group operation in terms of standard complexes. Nevertheless, the image  $\hat{h}(\Theta_{\SF})\subset \widehat{\frI}$ has a simple description; see \cite{DHSThomcob}*{Section~8}.  Indeed, 
\[
\hat{h}(\Theta_{\SF})=\{C(a_1,b_1,\ldots,a_k,b_k)\mid |b_i|\leq|b_{i-1}| \mbox{ and  }\sgn(b_i)=-\sgn(a_i)\}\subset \widehat{\frI}.
\]

 \subsection{Involutive knot Floer homology}

 Hendricks-Manolescu also constructed an involutive knot Floer homology $\cCFK(Y,K)$ for $K\subset Y$ a null-homologous knot, which, from our viewpoint, is a finitely-generated $\bF[\scU,\scV]$-complex with an endomorphism  $\iota_K$, with properties as follows. 
 
 Suppose that $(C_K,\d)$ is a free, finitely generated complex over the ring $\bF[\knotU,\knotV]$. There are two naturally associated maps
 \[
 \Phi,\Psi\colon C_K\to C_K,
 \]
 as follows. We write $\d$ as a matrix with respect to a free $\bF[\knotU,\knotV]$-basis of $C_K$. We define $\Phi$ to be the endomorphism obtained differentiating each entry of this matrix with respect to $\knotU$. We define $\Psi$ to be the endomorphism obtained by differentiating each entry with respect to $\knotV$. These maps naturally appear in the context of knot Floer homology, see \cites{SarkarMaslov, ZemQuasi, ZemCFLTQFT}. The maps $\Phi$ and $\Psi$ are independent of the choice of basis, up to $\bF[\knotU,\knotV]$-equivariant chain homotopy \cite{ZemConnectedSums}*{Corollary~2.9}.
 
 We say an $\bF$-linear map $F\colon C_K\to C_K'$ is \emph{skew-$\bF[\knotU,\knotV]$-equivariant} if 
 \[
 F\circ \knotV=\knotU\circ F\quad \text{and} \quad F\circ \knotU=\knotV\circ F.
 \]

 We may view a free complex over $\bF[\knotU,\knotV]$ also as an infinitely generated complex over $\bF[U]$, where $U$ acts diagonally via $U=\knotU \knotV$. Concretely, if $B=\{\xs_1,\dots, \xs_n\}$ is an $\bF[\knotU,\knotV]$-basis, then an $\bF[U]$-basis is given by the elements $\knotU^i\cdot  \xs_k$ and $\knotV^j\cdot  \xs_k$, ranging over all $i\ge 0$, $j\ge 0$ and $k\in \{1,\dots, n\}$.

 \begin{define}\label{def:iota-K-complex}
 	\begin{enumerate}
 		\item  An \emph{$\iota_K$-complex} $(C_K, \iota_K)$ is a finitely generated, free chain complex $C_K$ over $\bF[\knotU,\knotV]$, equipped with a skew-equivariant endomorphism $\iota_K$ satisfying 
 		\[
 		\iota_K^2\simeq \id+\Phi\Psi. 
 		\]
 		\item We say an $\iota_K$-complex $(C_K,\iota_K)$ is of \emph{$\Z H S^3$-type} if there are two $\Z$ valued gradings, $\gr_{\ws}$ and $\gr_{\zs}$, such that $\knotU$  and $\knotV$ have $(\gr_{\ws},\gr_{\zs})$-bigrading $(-2,0)$ and $(0,-2)$, respectively. We assume $\d$ has $(\gr_{\ws},\gr_{\zs})$-bigrading $(-1,-1)$, and that $\iota_K$ switches $\gr_{\ws}$ and $\gr_{\zs}$. Furthermore, we assume that $A:=\tfrac{1}{2}(\gr_{\ws}-\gr_{\zs})$ is integer valued. We call $A$ the \emph{Alexander} grading, and we call $\gr_{\ws}$ and $\gr_{\zs}$ the \emph{Maslov} gradings.  Writing $\cA_s\subset C_K$ for the subspace in Alexander grading $s$, we assume that there is a grading-preserving isomorphism $U^{-1} H_*(\cA_s)\iso \bF[U,U^{-1}]$ for all $s\in \Z$.

 	\end{enumerate}
 \end{define}
 
 In Definition~\ref{def:iota-K-complex}, an $\iota_K$-complex of $\Z H S^3$-type is equipped with two Maslov gradings, $\gr_{\ws}$ and $\gr_{\zs}$. We note that in the literature, usually one considers just the $\gr_{\ws}$-grading, which is referred to as the homological grading.  All $\iota_K$-complexes in this paper will be of $\Z H S^3$-type (as they arise as the complexes associated to knots in $S^3$).  For further details on the translation to other versions of knot Floer homology, see \cite[Section 1]{ZemAbsoluteGradings}.
 
 The tensor product of $\iota_K$-complexes has a slightly subtle definition:
 \[
 (C_K,\iota_K) \otimes (C_K',\iota_K') = (C_K\otimes_{\bF[\knotU,\knotV]} C'_K, \iota_K\otimes \iota_K'+\iota_K\Psi\otimes \iota_K'\Phi).
 \]

Local equivalence of $\iota_K$-complexes can defined much as for iota-complexes. See \cite{ZemConnectedSums}*{Section~2}.  For the present paper however, it is helpful to work with the (equivalent) definition that local equivalence of iota-complexes is the equivalence relation generated by declaring two $\iota_K$-complexes of $\Z H S^3$-type $C_1$ and $C_2$ locally equivalent if $C_1$ is an $\iota_K$-equivariant summand of $C_2$ (cf. \cite{HHStableEquivalence}). With respect to this definition one may form a local equivalence group of $\iota_K$-complexes; inverses are given by dualizing over $\bF[\scU,\scV]$. As previously, we let $-(C_K, \iota_K)$ denote the dual $\iota_K$-complex of $(C_K, \iota_K)$. 
 
 At present, it is very difficult to compute the $\iota_K$-complexes associated to most knots.  However, for $L$-space knots $K\subset S^3$, Hendricks-Manolescu \cite{HMInvolutive} observed that there is a unique choice of $\iota_K$ such that the knot Floer complex $\cCFK(K)$ is an $\iota_K$-complex.  In particular, the involutive knot Floer complex of an $L$-space knot $K$ is determined by the Alexander polynomial $\Delta_K(t)$ of $K$.
 
 \subsection{The surgery formula}
 Our main tool is the surgery formula from \cite{HHSZ}, which gives an expression for the involutive Heegaard Floer complex $(\CF^-(S^3_{+1}(K)),\iota)$ in terms of the involutive knot Floer complex of $K$.  
 
 We will only need a small part of the surgery formula.  For $K\subset S^3$,  let $A_0(K)$ denote the $\bF[U]$-subcomplex of $(\scU,\scV)^{-1}\cCFK(K)$, generated over $\bF$ by the monomials $\scU^i\scV^j\cdot \xs$ satisfying $A(\xs)+j-i=0$ with $i\geq 0$ and $j\geq 0$.  The $U$-action on $A_0(K)$ is given by $U=\scU\scV$.  Moreover, we can define a chain map $\iota\colon A_0(K)\to A_0(K)$ by $\iota(\xs)=\iota_K(\xs)$, since $\iota_K$  preserves $A_0(K)$.  It turns out (but is not obvious) that $(A_0(K),\iota_K)$ is an iota-complex (cf. \cite{HMInvolutive}*{Theorem~1.5} and \cite{HHSZ}*{Lemma~3.16}).  A consequence of the full surgery formula is:

 \begin{prop}[{\cite[Theorem 1.6]{HHSZ}}] \label{prop:surgery-formula}
 	The local equivalence class of $(\CF^-(S^3_{+1}(K)),\iota)$ is that of $(A_0(K),\iota_K)$. In particular, the $\iota$-local equivalence class of $(\CF^-(S_{+1}^3(K)),\iota)$ depends only on the $\iota_K$-local class of $(\cCFK(K),\iota_K)$.
 \end{prop}

\section{Computation of the almost iota-complex of $Y_n$} \label{section-3-iotacomplex}

In this section we give a computation of the almost iota-complex associated to the manifold $Y_n=S^3_{+1}(T_{2,3} \# -2T_{2n, 2n+1} \# T_{2n, 4n+1})$ for $n\geq 3$ odd. We start by describing the knot Floer homology of the two torus knots $T_{2n, 2n+1}$ and $T_{2n, 4n+1}$, followed by computing and simplifying several tensor products.

\subsection{The knot Floer homology of two families of torus knots} \label{subsec:torus}

In this subsection we compute the $\iota_K$-complexes associated to the torus knots $T_{2n, 2n+1}$ and $T_{2n,4n+1}$.

Let $\cC_n$ denote the $\bF[\scU,\scV]$ complex in Figure \ref{fig:Cn} generated by elements $x_{k}$ such that $-2n+2 \leq k \leq 2n-2$ with $k$ even and $y_{\ell}$ such that $-2n+1 \leq \ell \leq 2n-1$ with $\ell$ odd, with nonzero differentials given by
\begin{align*}
\partial(x_{k}) = \scV^{c_{2n-1+k}}y_{k-1} + \scU^{c_{2n+k}}y_{k+1}
\end{align*}
\noindent determined by the symmetric sequence of positive integers $(c_1,c_2,\dots, c_{4n-3},c_{4n-2}) = (1, 2n-1, 2, 2n-2, \dots, 2n-2,2,2n-1,1)$.

\begin{figure}
\begin{tikzcd}[column sep={.8cm,between origins}, labels=description, row sep=1cm]
y_{1-2n}&& y_{3-2n}&& \cdots&& y_{-1}&& y_{1}&&\cdots&& y_{2n-3}&&y_{2n-1}\\
&x_{2-2n}\ar[ul, "\scV"]\ar[ur, "\scU^{2n-1}"]&&x_{4-2n}\ar[ul, "\scV^2"] \ar[ur, "\scU^{2n-2}"]&& \cdots\ar[ur] \ar[ul]  && x_{0} \ar[ul, "\scV^{n}"] \ar[ur, "\scU^{n}"]&& \cdots \ar[ur] \ar[ul]  && x_{2n-4} \ar[ul, "\scV^{2n-2}"] \ar[ur, "\scU^{2}"]&& x_{2n-2} \ar[ul, "\scV^{2n-1}"] \ar[ur, "\scU"]
\end{tikzcd}
\caption{The complex $\cC_n$.}\label{fig:Cn}
\end{figure}

Likewise, let $\cD_n$ denote the complex defined similarly using the symmetric string of positive integers $(c_1, \cdots, c_{8n-4})$ given by
\[
(1,2n-1,1,2n-1,2,2n-2,2,2n-2,\cdots ,2n-2,2,2n-2,2,2n-1,1,2n-1,1)
\]
\noindent with generators $w_{k}$ such that $3-4n \leq k \leq 4n-3$ with $k$ odd and $z_{\ell}$ such that $2-4n \leq \ell \leq 4n-2$ with $\ell$ even, and nonzero differentials given by
\[
\partial(w_k) = \scV^{c_{4n-2+k}}z_{k-1} + \scU^{c_{4n-1+k}}z_{k+1}.
\]
\noindent See Figure \ref{fig:Dn} for a depiction of this staircase.

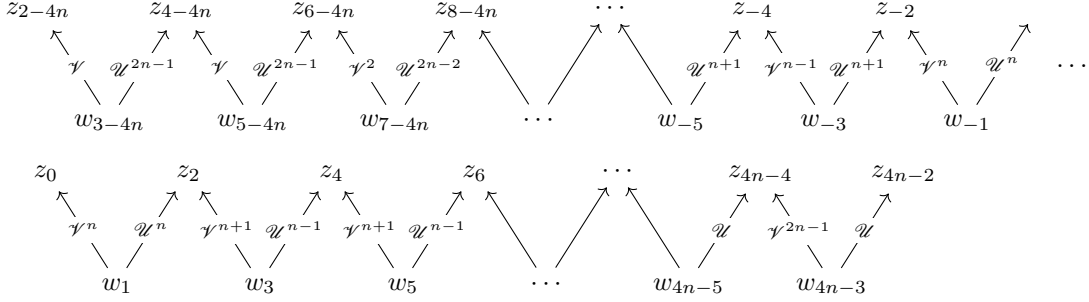
\begin{figure}
\[
\begin{tikzcd}[column sep={.95cm,between origins}, labels=description, row sep=1cm]
z_{2-4n}&& z_{4-4n}&& z_{6-4n}&& z_{8-4n}&& \cdots&& z_{-4}&& z_{-2}&&\,\\
& w_{3-4n} \ar[ul, "\scV"]\ar[ur, "\scU^{2n-1}"]&& w_{5-4n}\ar[ul, "\scV"]\ar[ur, "\scU^{2n-1}"]&& w_{7-4n} \ar[ul, "\scV^2"]\ar[ur, "\scU^{2n-2}"]&& \cdots \ar[ur] \ar[ul] && w_{-5}\ar[ul]\ar[ur, "\scU^{n+1}"] && w_{-3} \ar[ul, "\scV^{n-1}"]\ar[ur, "\scU^{n+1}"] && w_{-1} \ar[ul, "\scV^n"] \ar[ur, "\scU^n"]
\end{tikzcd} \dots
\]
\[
\begin{tikzcd}[column sep={.95cm,between origins}, labels=description, row sep=1cm]
z_{0}&& z_2&& z_4&& z_6&& \cdots&& z_{4n-4}&& z_{4n-2}&&\,\\
& w_1 \ar[ul, "\scV^n"]\ar[ur, "\scU^n"]&& w_3 \ar[ul, "\scV^{n+1}"]\ar[ur, "\scU^{n-1}"]&&w_5  \ar[ul, "\scV^{n+1}"]\ar[ur, "\scU^{n-1}"]&& \cdots \ar[ul] \ar[ur] && w_{4n-5}\ar[ul]\ar[ur, "\scU"] && w_{4n-3} \ar[ul, "\scV^{2n-1}"] \ar[ur, "\scU"]
\end{tikzcd}
\]

\caption{The complex $\cD_n$. Note that the staircase continues onto the second row of the figure, such that $\partial(w_{-1})= \scV^n z_{-2} + \scU^n z_0$.} \label{fig:Dn}
\end{figure}

\begin{prop} \label{prop:torusknots} For $n$ odd, the knot Floer homology $\cCFK(T_{2n,2n+1})$ is chain homotopy equivalent to the complex $\cC_n$, and the knot Floer homology $\cCFK(T_{2n,4n+1})$ is chain homotopy equivalent to the complex $\cD_n$. In both cases the involution $\iota_K$ is given by the natural reflection which interchanges the bigradings.
\end{prop} 
 
By \cite{OSlens} and \cite{HMInvolutive}, the $\iota_K$-complex associated to a torus knot is determined by its Alexander polynomial, so it suffices to compute the Alexander polynomials of $T_{2n,2n+1}$ and $T_{2n, 4n+1}$ for $n$ odd.
Suppose that $(c_1,c_2,\dots, c_{2k-1},c_{2k})$ is a symmetric sequence of positive integers. We will define
\[
\Delta(c_1,c_2,\dots, c_{2k-1},c_{2k}):=1-t^{c_1}+t^{c_1+c_2}-t^{c_1+c_2+c_3}+\cdots -t^{c_1+\cdots+c_{2k-1}}+t^{c_1+\cdots+c_{2k}}.
\]
 
\begin{lem}\label{lem:alexander-poly-1} The Alexander polynomial of $T_{2n,2n+1}$ is given by the formula
\begin{equation}
\begin{split}
\Delta_{T_{2n,2n+1}}(t)&=\Delta(1,2n-1,2,2n-2,\dots, 2n-1,1)\\
&=1-t+t^{2n}-t^{2n+2}+t^{4n}-\cdots +t^{2n(2n-2)}-t^{2n(2n-1)-1}+t^{2n(2n-1).}
\end{split}
\label{eq:alexander-polynomial-T2n2n+1-statement}
\end{equation}
\end{lem}
\begin{proof}
Write $\Delta=\Delta(1,2n-1,2,2n-2,\dots, 2n-1,1)$.
The Alexander polynomial is given by 
\[
\Delta_{T_{2n,2n+1}}=\frac{(t^{2n(2n+1)}-1)(t-1)}{(t^{2n+1}-1)(t^{2n}-1)}.
\]

\noindent Rearranging, it becomes sufficient to show that
\[
\frac{(t^{2n})^{2n+1}-1}{t^{2n}-1}=\Delta\cdot \frac{t^{2n+1}-1}{t-1}.
\]
Expanding this out, our desired relation becomes
\begin{equation}
\sum_{i=0}^{2n} t^{2n i}=\Delta\cdot \sum_{i=0}^{2n} t^i. \label{eq:desired-relation-T2n2n+1}
\end{equation}

It is helpful to state a simple algebraic fact. Note that if $N$ and $M$ are positive integers, and $\{a_j\}_{j\in \Z}$ is a sequence which is zero for $j\not\in \{0,\dots, N\}$, then
\[
\left(\sum_{j=0}^N a_j t^j\right)\left(\sum_{i=0}^M t^i\right)=\sum_{j=0}^{N+M} (a_{j-M}+a_{j-M+1}+\cdots a_{j})t^j. 
\]
In particular, if we write $a_0,\dots, a_{2n(2n-1)}$ for the coefficients of $\Delta$ (and set $a_j=0$ for $j\not\in \{0,\dots, 2n(2n-1)\}$),  then the $t^j$ coefficient of the right hand side of equation~\eqref{eq:desired-relation-T2n2n+1} is 
\begin{equation}
a_{j-2n}+a_{j-2n+1}+\cdots+ a_j. \label{eq:sum-of-coefficients-a-j}
\end{equation}
However, by examining the description of $\Delta$ given in equation~\eqref{eq:alexander-polynomial-T2n2n+1-statement} it is easy to verify that equation~\eqref{eq:sum-of-coefficients-a-j} is $1$ if $j=2nk$ for some $k\in \{0,\dots, 2n\}$, and is 0 otherwise. This verifies equation~\eqref{eq:desired-relation-T2n2n+1}, and completes the proof. \end{proof}

 \begin{lem} \label{lem:alexander-poly-2} The Alexander polynomial of $T_{2n,4n+1}$ satisfies
 \[
\Delta_{T_{2n,4n+1}}(t)=\Delta(1,2n-1,1,2n-1,2,2n-2,2,2n-2,\dots, 2n-1,1,2n-1,1).
 \]
 \end{lem}
\begin{proof} 
The proof is in much the same spirit as the proof of Lemma~\ref{lem:alexander-poly-1}. Let $\Delta$ denote $\Delta(1,2n-1,1,2n-1,2,2n-2,2,2n-2,\dots, 2n-1,1,2n-1,1)$. Using the definition of the Alexander polynomial and rearranging terms, as in Lemma~\ref{lem:alexander-poly-1}, it is sufficient to show that
\[
\frac{(t^{2n})^{4n+1}-1}{t^{2n}-1}=\Delta \cdot \frac{t^{4n+1}-1}{t-1},
\]
which we expand to 
\[
\sum_{i=0}^{4n} t^{2ni}=\Delta \cdot \sum_{i=0}^{4n} t^i.
\]
Following the argument of Lemma~\ref{lem:alexander-poly-1}, it is sufficient to show that if $a_j$ denote the coefficients of $\Delta$, then $a_{j-4n}+\cdots +a_{j}$ is 1 if $j=2nk$, for some $k\in \{0,\dots 4n\}$, and is 0 otherwise. This is straightforward to verify.
\end{proof}

\begin{proof}[Proof of Proposition \ref{prop:torusknots}]
By \cite{OSlens} and \cite{HMInvolutive}, the $\iota_K$-complex of a torus knot is determined by its Alexander polynomial. The Alexander polynomials computed in Lemmas \ref{lem:alexander-poly-1} and \ref{lem:alexander-poly-2} correspond to the staircases $\cC_n$ and $\cD_n$ respectively, with $\iota_K$ given by the natural involution in each case.\end{proof}

\subsection{The $\iota_K$-complex associated to $-2T_{2n, 2n+1} \# T_{2n, 4n+1}$}

In this subsection we compute the $\iota_K$-complex associated to the connect sum of torus knots $-2T_{2n, 2n+1} \# T_{2n, 4n+1}$ up to $\iota_K$-local equivalence.

\subsubsection{The $\iota_K$-local equivalence class of $T_{2n, 2n+1}\# T_{2n, 2n+1}$}
As in Section \ref{subsec:torus}, let $\cC_n$ denote the complex of $T_{2n,2n+1}$ for $n$ odd which appears in Figure \ref{fig:Cn}, and let $\cD_n$ denote the complex associated to $T_{2n,4n+1}$ which appears in Figure \ref{fig:Dn}. We first consider $\cX_n:=\cC_n\otimes \cC_n$. We will choose a new basis for $\cX_n$ with respect to which our complex decomposes as a direct sum of $\cY_n\oplus \cZ_n$, as follows. The subset $\cY_n$ is generated by the basis elements appearing in Figure~\ref{fig:Yn}.

\begin{figure}[h]
\[
\begin{tikzcd}[column sep={1cm,between origins}, labels=description, row sep=1cm]
y_{1-2n}y_{1-2n}&& y_{1-2n}y_{3-2n}&& y_{3-2n}y_{3-2n}&& y_{3-2n}y_{5-2n}&& \cdots&& y_{-3}y_{-1}&& y_{-1} y_{-1}&&\,\\
& y_{1-2n} x_{2-2n} \ar[ul, "\scV"]\ar[ur, "\scU^{2n-1}"]&& x_{2-2n}y_{3-2n}\ar[ul, "\scV"]\ar[ur, "\scU^{2n-1}"]&& y_{3-2n}x_{4-2n} \ar[ul, "\scV^2"]\ar[ur, "\scU^{2n-2}"]&& \cdots \ar[ur] \ar[ul] && y_{-3}x_{-2}\ar[ul]\ar[ur, "\scU^{n+1}"] && x_{-2} y_{-1} \ar[ul, "\scV^{n-1}"]\ar[ur, "\scU^{n+1}"] && x_{0}y_{-1} \ar[ul, "\scV^n"] \ar[ur, "\scU^n"]
\end{tikzcd} \cdots
\]
\[
\begin{tikzcd}[column sep={1cm,between origins}, labels=description, row sep=1cm]
y_{1}y_{-1}&& y_{1}y_{1}&& y_{3}y_{1}&& y_{3}y_{3}&& \cdots&& y_{2n-1}y_{2n-3}&& y_{2n-1}y_{2n-1}&&\,\\
& y_{1}x_{0} \ar[ul, "\scV^n"]\ar[ur, "\scU^n"]&& x_{2}y_{1}\ar[ul, "\scV^{n+1}"]\ar[ur, "\scU^{n-1}"]&&y_{3}x_{2}  \ar[ul, "\scV^{n+1}"]\ar[ur, "\scU^{n-1}"]&& \cdots \ar[ul] \ar[ur] && x_{2n-2}y_{2n-3}\ar[ul]\ar[ur, "\scU"] && y_{2n-1}x_{2n-2} \ar[ul, "\scV^{2n-1}"] \ar[ur, "\scU"]
\end{tikzcd}
\]
\[
\begin{tikzcd}[column sep={1.55cm,between origins}, labels=description, row sep=1cm]
&y_{1}y_{-1}+y_{-1}y_1&\,\\
y_{-1}x_{0}+x_{0}y_{-1} \ar[ur, "\scU^n"]&& y_1x_{0}+x_{0}y_1 \ar[ul, "\scV^n"]\\
& x_{0}x_{0}\ar[ul, "\scV^n"]\ar[ur, "\scU^n"]
\end{tikzcd}
\]
\caption{The subcomplex $\cY_n\subset \cX_n$. Note that the top two rows form a staircase complex, such that $\partial(x_0y_{-1})=\scV^n y_{-1}y_{-1}+\scU^n y_1y_{-1}$.} \label{fig:Yn}
\end{figure}

Observe that in the staircase summand of the subcomplex $\cY_n$, the pattern of the construction changes at the basis element $y_1y_1$. Namely, traveling left to right in Figure \ref{fig:Yn} along the top row, we increase the second index of the generators $y_iy_j$, followed by the first. Along the second row, we increase the first index of the generators $y_iy_j$, followed by the second. This complex is equipped with the involution $\iota_K$ arising from the tensor product, which in particular sends
\begin{align*}
\iota_K(x_0x_0) &= x_0x_0 + \scU^{n-1}\scV^{n-1} y_{1}y_{-1} \\
\iota_K(y_1y_{-1}) &= y_{1}y_{-1}+(y_1y_{-1} + y_{-1}{y_1}) \\
\iota_K(y_1y_{-1} + y_{-1}{y_1}) &= y_1y_{-1} + y_{-1}{y_1} \\
\iota_K(y_{-1}x_0+x_0y_{-1}) &= y_1x_0+x_0y_1 \\
\iota_K(y_1x_0 + x_0y_1) &= y_{-1}x_0+x_0y_{-1} \\
\iota_K(x_0y_{-1})&= y_1x_0 + (y_1x_0+x_0y_1)\\
\iota_K(y_1x_0) &= x_0y_{-1} + (y_{-1}x_0+x_0y_{-1})
\end{align*}
\noindent and is otherwise a reflection.

Before defining the summand $\cZ_n$, we make a few preliminary observations about gradings. Firstly, we note that
\[
A(y_{i+2})=A(y_i)+2n,
\]
for all odd $i$. As a consequence, if $i$ and $j$ are odd, then
\[
A(y_{i}y_{j})=A(y_{i+2}y_{j-2}).
\]
In particular, if $i$, $j$ are odd, then there is an $\gamma_{i,j}\in \Z$ such that
\[
y_{i}y_{j}+(\scU \scV)^{\g_{i,j}} y_{i+2} y_{j-2}
\]
has homogeneous $(\gr_{\ws},\gr_{\zs})$-bigrading. It is not hard to compute that if $i<j$, then $\g_{i,j}\ge 0$.

Suppose that $i$ and $j$ are even and $i<j$. By considering the differential applied to $x_{i}x_j$ and using the fact that the $\scU$ powers in $\d x_i$ decrease as we increase the index of $x_i$, we see that if $i<j$, then there is an $\alpha_{i,j}\ge  0$ such that
\[
\scU^{\alpha_{i,j}}y_{i+1}x_j+ x_i y_{j+1}
\]
has homogeneous $(\gr_{\ws},\gr_{\zs})$-bigrading. Entirely analogously, if $i<j$, then there is a $\beta_{i,j}\ge 0$ so that
\[
y_{i-1}x_j+\scV^{\beta_{i,j}}x_iy_{j-1}
\]
has homogeneous $(\gr_{\ws},\gr_{\zs})$-bigrading.

We now describe the summand $\cZ_n$. The generators have the following form:

\begin{enumerate}[label=($Z$-\arabic*), ref=$Z$-\arabic*]
\item\label{Z-1} If $i$ and $j$ are both odd and $i\neq \pm j$, then $y_iy_j+y_jy_i$ is a generator of $\cZ_n$.
\item\label{Z-2} If $i$ is odd, $j$ is even, and $j\neq -i\pm 1$, then $y_ix_j+x_jy_i$ is a generator of $\cZ_n$.
\item \label{Z-3}  If $i>0$ is even and non-zero, write $i=2n-2k$ for some $n> k\ge 1$. Then $\cZ_n$ has a generator
\[
x_ix_i+k(2n-k)\scU^{k-1}\scV^{2n-k-1}y_{i-1}y_{i+1}.
\]
\item \label{Z-4} If $i<0$ is even then $\cZ_n$ has a generator $x_ix_i.$
\item\label{Z-5} If $i$ and $j$ are even with $i< j$, then  
\[
x_ix_j+k_i(2n-k_j) \scU^{k_i-1} \scV^{2n-k_j-1} y_{i-1}y_{j+1}
\] 
is a generator of $\cZ_n$, where $i=2n-2k_i$ and $j=2n-2k_j$. 
\item \label{Z-6} If $i$ and $j$ are even and $i>j$, then $x_ix_j$ is a generator of $\cZ_n$.
\item\label{Z-7} If $i$ and $j$ are odd and $i<j-2$, then 
\[
y_iy_j+(\scU\scV)^{\gamma_{i,j}}   y_{i+2}y_{j-2}
\]
 is a generator of $\cZ_n$. 
 \item\label{Z-8} If $i>2$ is odd, then
\[
y_iy_{-i}+(\scU\scV)^{\gamma_{-i,i}}  y_{i-2}y_{-i+2}
\]
 is a generator.
\item\label{Z-9} If $i$ and $j$ are even and $i<j$, then following are generators of $\cZ_n$:
\begin{enumerate}
\item \label{Z-9a}  $x_iy_{j+1}+\scU^{\a_{i,j}}y_{i+1}x_j$;
\item \label{Z-9b} $x_{-i} y_{-j-1}+ \scV^{\b_{-i,-j}}y_{-i-1}x_{-j}$.
\end{enumerate}
\item \label{Z-10} For even $j>0$, the following are generators of $\cZ_n$:
\begin{enumerate}
\item \label{Z-10a} $y_{j+1} x_{-j}+ \scV^{\beta_{j,-j}} x_j y_{-j+1}$;
\item \label{Z-10b} $y_{-j-1} x_{j}+\scU^{\a_{-j,j}} x_{-j}y_{j-1}$.
\end{enumerate}
\item\label{Z-11} For even $j>0$, the following are generators of $\cZ_n$:
\begin{enumerate}
\item\label{Z-11a} $y_{j-1}x_{-j}+\scU^{\a_{-j,j-2}}x_{j-2}y_{-j+1}$;
\item\label{Z-11b}  $y_{-j+1}x_j+\scV^{\b_{-j+2,j}} x_{-j+2} y_{j-1}$.
\end{enumerate}
\end{enumerate}

In the following, we use
\[
(\iota_K\otimes \iota_K)\circ (\id\otimes\id+\Psi \otimes \Phi)
\]
as our model of the involution.

\begin{lem}\label{lem-cn-times-cn} $\cZ_n$ and $\cY_n$ satisfy the following:
\begin{enumerate}
\item\label{claim-Z-1} $\cZ_n$ and $\cY_n$ are free.
\item\label{claim-Z-2} $\cX_n\iso \cY_n\oplus \cZ_n$.
\item\label{claim-Z-3} $\d \cZ_n\subset \cZ_n$ and $\d \cY_n\subset \cY_n$
\item\label{claim-Z-4} $\iota_K \cZ_n\subset \cZ_n$ and $\iota_K \cY_n\subset \cY_n$.
\end{enumerate}
In particular, $\cX_n$ is $\iota_K$-locally equivalent to $\cY_n$.
\end{lem}
\begin{proof}
To prove \eqref{claim-Z-1} and~\eqref{claim-Z-2}, we will first show that $\cX_n=\cY_n+\cZ_n$, and then we will show that the generating set obtained by concatenating the obvious basis of $\cY_n$ with the basis for $\cZ_n$ above gives a generating set of $\cX_n$ of the correct number of elements.  In particular, this will imply that $\cZ_n$ is free since it has a generating set with no linear relations.

We first address $\cX_n=\cY_n+\cZ_n$. Suppose $i$ and $j$ are both odd. Note $y_iy_i$ is in $\cY_n$ so we may assume that $i\neq j$. Consider the case $i\neq -j$. By adding~\eqref{Z-1}, it is sufficient to consider $i<j$. By adding~\eqref{Z-7}, we reduce to the case of $y_iy_i$ or $y_iy_{i+2}$ which are either in $\cY_n$, or are a sum of an element in $\cY_n$ with an element of \eqref{Z-1}.  Now consider $y_iy_{-i}$. By adding elements~\eqref{Z-8}, we reduce to the case of $y_{1}y_{-1}$ and $y_{-1}y_1$, which are both in $\cY_n$. We now consider elements $x_i y_j$ and $y_j x_i$. Note that if $|i-j|=1$, then either $x_iy_j$ is in $\cY_n$, or $x_i y_j$ plus an element ~\eqref{Z-2} is in $\cY_n$. The same holds for $y_jx_i$. Next, we consider an arbitrary $x_iy_j$. Using ~\eqref{Z-9}, we may relate $x_iy_j$ with sums of $x_ny_m$ and $y_mx_n$ with $|m-n|<|i-j|$. Hence, by induction, it suffices to show that we can do the same to $y_j x_i$. If $j\neq -i\pm 1$, then we use~\eqref{Z-2} to relate $y_jx_i$ with $x_i y_j$, and apply the previous argument. If $j=-i\pm 1$, then we use~\eqref{Z-10} or~\eqref{Z-11}. This shows that all $x_i y_j$ and $y_jx_i$ are in the span. Finally, each $x_ix_j$ is a sum of generators~\eqref{Z-4},~\eqref{Z-5} and~\eqref{Z-6}, as well as the above terms. Hence $\cX_n=\cY_n+\cZ_n$.

We now show that the generating set obtained by concatenating $\cY_n$ and $\cZ_n$ has the same cardinality as the rank of $\cX_n$, which implies that $\cZ_n$ is free and $\cX_n\iso \cY_n\oplus \cZ_n$. Firstly, 
\[
\rank(\cX_n)=16 n^2-8n+1\quad \text{and} \quad \rank(\cY_n)=8n+1.
\]
Similarly, $\cZ_n$ has $2n^2-2n$ generators of type ~\eqref{Z-1}, $4n^2-6n+2$ generators of type ~\eqref{Z-2}, $4n^2-4n$ generators of types \eqref{Z-3}, \eqref{Z-4}, \eqref{Z-5} or \eqref{Z-6}, $2n^2-2n$ generators of type~\eqref{Z-7} and~\eqref{Z-8}, and $4n^2-6n+2$ generators of type \eqref{Z-9}, and $4n-4$ generators of type~\eqref{Z-10} or~\eqref{Z-11}. Hence, we have a generating set of $\cZ_n$ with $16n^2-16n$ generators. Concatenating these generating sets gives a generating set of $\cC_n\otimes \cC_n$ with rank $16n^2-8n+1$, which must be a basis.

We now prove~\eqref{claim-Z-3}. Clearly $\d \cY_n\subset \cY_n$, so we focus on $\cZ_n$. On \eqref{Z-1}, $\d$ vanishes. The map $\d$ sends elements of type~\eqref{Z-2} to a sum of two elements of ~\eqref{Z-1}. Elements~\eqref{Z-3} and~\eqref{Z-4} are mapped to sums of ~\eqref{Z-2}. Basis elements in~\eqref{Z-5} are mapped to a sum of~\eqref{Z-9a} and~\eqref{Z-9b}. Basis elements \eqref{Z-6} are as follows. If $|j+i-1|>1$, they are mapped to a sum of ~\eqref{Z-9a} and~\eqref{Z-9b}. If $i+j=2$, they are mapped to a sum of \eqref{Z-9a}, ~\eqref{Z-9b} and~\eqref{Z-2}. If $i+j=-2$, they are mapped to a sum of \eqref{Z-9b} and~\eqref{Z-11a}. The differential vanishes on~\eqref{Z-7} and~\eqref{Z-8}. 
Elements~\eqref{Z-9a} are mapped to elements~\eqref{Z-7}. Elements~\eqref{Z-10} are mapped to a sum of~\eqref{Z-1} and~\eqref{Z-7} if $i\neq -j$, or \eqref{Z-8} if $i=-j$. Elements~\eqref{Z-10a} are mapped to~\eqref{Z-8}. Elements~\eqref{Z-10b} are mapped to~\eqref{Z-7}. Elements~\eqref{Z-11a} are mapped to a sum of~\eqref{Z-1} and~\eqref{Z-7}. Finally~\eqref{Z-11b} is mapped to ~\eqref{Z-7}.

We now prove~\eqref{claim-Z-4}. Clearly $\iota_K \cY_n\subset \cY_n$, so we focus on $\cZ_n$. The map $\iota_K$ sends elements \eqref{Z-1} to elements ~\eqref{Z-1}. Similarly elements ~\eqref{Z-2} are sent to elements~\eqref{Z-2}. Elements~\eqref{Z-3} are sent to elements~\eqref{Z-4}. Elements~\eqref{Z-4} are sent to the sum of an element~\eqref{Z-3} and an element~\eqref{Z-1}. Similarly elements~\eqref{Z-5} are sent to elements~\eqref{Z-6}, while elements~\eqref{Z-6} and sent to sums of~\eqref{Z-5} and~\eqref{Z-7}.  Generators~\eqref{Z-7}  with $i\neq -j$ are sent to a sum ~\eqref{Z-7} and two elements of~\eqref{Z-1}. Generators~\eqref{Z-7} with $i=-j$ are interchanged with generators~\eqref{Z-8}. Elements ~\eqref{Z-9a} and~\eqref{Z-9b} are interchanged. Elements~\eqref{Z-10a} and~\eqref{Z-10b} are interchanged. Similarly elements~\eqref{Z-11a} and~\eqref{Z-11b} are interchanged. 
\end{proof}

\subsubsection{The $\iota_K$-local equivalence class of $-2T_{2n, 2n+1}\# T_{2n, 4n+1}$}

In this section, we compute the $\iota_K$-local equivalence class of $\cCFK(-2T_{2n, 2n+1}\# T_{2n, 4n+1})$.

We begin by introducing a new complex, called the \emph{box complex.} Let $\cB_n$ denote the knot-like complex in Figure \ref{fig:Bn} with five generators $v,$ $u,$ $s_1,$ $s_{-1}$, $s_0$, with differential
\[
\d v=0,\quad \d s_0=\scV^n s_{-1} +\scU^n s_1, \quad \d s_{-1}=\scU^n u, \quad \d s_1=\scV^n u ,\quad \text{and} \quad \d u=0.
\]
The $\gr=(\gr_{\ws},\gr_{\zs})$-bigradings are as follows:
\[
\begin{split}
\gr(v)&=(0,0),\\
\gr(s_0)&=(2-2n,2-2n),\\
\gr(s_{-1})&=(1-2n,1)\\
\gr(s_1)&=(1,1-2n) \quad \text{and}\\
\gr(u)&=(0,0).
\end{split}
\]
The involution on $\cB_n$ is as follows:
\[
\begin{split}
\iota_K(v)&=v+u\\
\iota_K(s_0)&=s_0+\scU^{n-1} \scV^{n-1} v\\
\iota_K(s_{-1})&=s_1\\
\iota_K(s_1)&=s_{-1}\\
\iota_K(u)&=u.
\end{split}
\]

\begin{figure}
\begin{minipage}{.2\textwidth}
\[
v \quad \begin{tikzcd}[row sep=1cm, column sep=1cm, labels=description] s_{-1}\ar[r, "\scU^n"] & u\\
s_0\ar[u,"\scV^n"] \ar[r, "\scU^n"]& s_1 \ar[u, "\scV^n"]
\end{tikzcd}
\]\end{minipage}\hspace{2cm}\begin{minipage}{.2\textwidth}
\[
v^\vee \quad \begin{tikzcd}[row sep=1cm, column sep=1cm, labels=description] s_{-1}^\vee \ar[d,"\scV^n"] & u^\vee\ar[l, "\scU^n"] \ar[d, "\scV^n"]\\
s^\vee_0 & s^\vee_1 \ar[l, "\scU^n"] 
\end{tikzcd}
\]\end{minipage}
\caption{The box complex $\cB_n$ and its dual $\cB_n^{\vee}$.}\label{fig:Bn}
\end{figure}

We will also be interested in the dual complex $\cB_n^\vee$, which is generated by $v^{\vee}, s^{\vee}_0, s^{\vee}_{-1}, s^{\vee}_{1}, u^{\vee}$ with gradings

\[
\begin{split}
\gr(v^{\vee})&=(0,0),\\
\gr(s^{\vee}_0)&=(2n-2,2n-2),\\
\gr(s^{\vee}_{-1})&=(-1,2n-1)\\
\gr(s^{\vee}_1)&=(2n-1,-1) \quad \text{and}\\
\gr(u^{\vee})&=(0,0).
\end{split}
\]
and involution
\[
\begin{split}
\iota_K(v^{\vee})&=v^{\vee}+\scU^{n-1} \scV^{n-1}s^{\vee}_0\\
\iota_K(s^{\vee}_0)&=s^{\vee}_0\\
\iota_K(s^{\vee}_{-1})&=s^{\vee}_1\\
\iota_K(s^{\vee}_1)&=s^{\vee}_{-1}\\
\iota_K(u^{\vee})&=u^{\vee}+v^{\vee}.
\end{split}
\]

\begin{prop} \label{prop-three-is-box} The $\iota_K$-complex of $(\cCFK(-2T_{2n, 2n+1}\# T_{2n, 4n+1}), \iota_K)$ is $\iota_K$-locally equivalent to the complex $\cB_n^\vee$ with the involution described above. \end{prop}

Recall that by Lemma \ref{lem-cn-times-cn}, $\cCFK(2T_{2n, 2n+1})$ is $\iota_K$-locally equivalent to the complex $\cY_n$ of Figure \ref{fig:Yn}. Moreover, by Proposition \ref{prop:torusknots}, $\cCFK(T_{2n,4n+1})$ is $\iota_K$-locally equivalent to the complex $\cD_n$ of Figure \ref{fig:Dn}. Our proof of Proposition \ref{prop-three-is-box} proceeds by demonstrating that the $\iota_K$-complex $\cY_n$ is $\iota_K$-locally equivalent to $\cD_n \otimes \cB_n$.

Indeed, we prove a general lemma about the tensor product of (positive) staircase complexes with an even number of steps and box complexes. Let $k$ be an even number, and let $\cS$ be a staircase complex with generators $x_{j}$ such that $-k+1 \leq j \leq k-1$ with $j$ odd, and $y_{i}$ such that $-k \leq i \leq k$ with $i$ even. Let the differentials 
\begin{align*}
\partial(x_{j}) = \scV^{c_{k+j}}y_{j-1} + \scU^{c_{k+j+1}}y_{j+1}
\end{align*}
\noindent be specified by a symmetric sequence of positive integers $(c_1,c_2,\dots, c_{2k-1},c_{2k})$ with the property that $c_k=c_{k+1}=n$. Most importantly, $\cS$ has an even number of steps and the central arrows with target $y_0$ are both weighted by $n$, so that
\[
\d x_{-1}=\scV^{c_{k-1}} y_{-2}+\scU^n y_0\quad \text{and} \quad \d x_1=\scU^{c_{k-1}} y_2+\scV^n y_0.
\]
\noindent (Recall that $c_{k-1} = c_{k+2}$.) We will compute the $\iota_K$-local equivalance class of $\cS\otimes \cB_n$ for any staircase of this form. Similarly to the methods of the previous subsection, we construct an $\iota_K$-equivariant splitting
\[
\cS\otimes \cB_n\iso \cY\oplus \cW
\]
\noindent into two summands $\cY$ and $\cW$, which we now describe. The complex $\cY$, which is the simpler of the two, appears in Figure \ref{fig:Ygeneral}. The complex $\cW$ has the following generators:

\begin{figure}
\[
\begin{tikzcd}[column sep={1.1cm,between origins}, labels=description, row sep=1cm]
y_{-k}v&& \cdots&& y_{-2}v&& y_0 v&& y_{2}(v+u)&&\cdots&& y_k(v+u)\\
&x_{-k+1}v\ar[ul, "\scV^{c_1}"]\ar[ur, "\scU^{c_2}"]&&\cdots \ar[ul] \ar[ur]&& x_{-1}v \ar[ur, "\scU^n"] \ar[ul, "\scV^{c_{k-1}}"]  && x_{1}(v+u)+ y_0s_1 \ar[ul, "\scV^{n}"] \ar[ur, "\scU^{c_{k-1}}"]&& \cdots \ar[ur] \ar[ul]  && x_{k-1}(v+u) \ar[ul, "\scV^{c_2}"] \ar[ur, "\scU^{c_1}"]
\end{tikzcd}
\]
\[
\begin{tikzcd}[row sep=1cm, column sep=1cm, labels=description] y_0s_{-1}\ar[r, "\scU^n"] & y_0u\\
y_0s_0\ar[u,"\scV^n"] \ar[r, "\scU^n"]& y_0s_1 \ar[u, "\scV^n"]
\end{tikzcd}
\]

\caption{The complex $\cY$. Note that on the bottom row of the staircase complex, the terms to the right of $x_1(v+u)+y_0s_1$ are all of the form $x_i(v+u)$.}\label{fig:Ygeneral}
\end{figure}
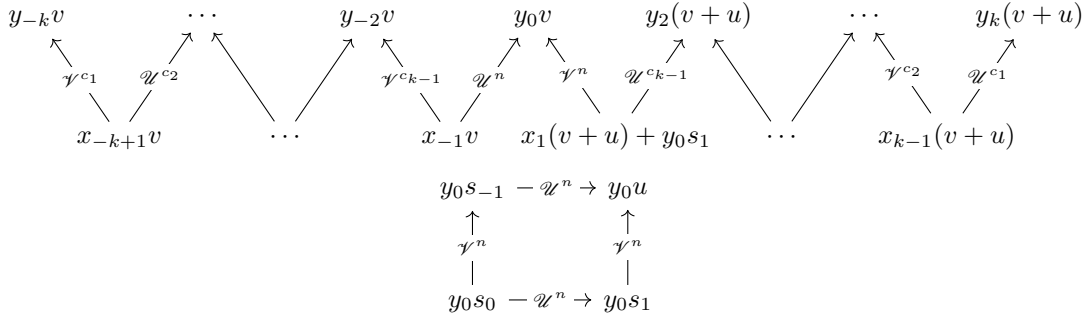

\begin{enumerate}[label=($W$-\arabic*), ref=$W$-\arabic*]
\item\label{W-1} For even $i\neq 0$, the element $y_i u$.
\item \label{W-2} For odd $i\not\in \{1,-1\}$, the element $x_i u$.
\item \label{W-3} The elements 
\[
x_{-1}u+y_0s_{-1}\quad \text{and} \quad x_1 u+y_0s_{1}.
\]
\item  \label{W-4} For $i>0$ even, the elements 
\[
y_i s_{-1}, \quad y_is_{1}\quad \text{and} \quad y_i(s_0+\scU^{n-1}\scV^{n-1}v).
\]
\item  \label{W-5} For $i<0$ even, the elements 
\[
y_i s_{-1}, \quad y_is_{1}\quad \text{and} \quad y_is_0.
\]
\item \label{W-6} For $i>1$ odd, then 
\[
x_is_{-1}, \quad x_is_1 \quad \text{and}  \quad  x_i(s_0+\scU^{n-1}\scV^{n-1} v).
\]
\item \label{W-7} For $i<-1$ odd, then 
\[
x_is_{-1}, \quad x_is_1 \quad \text{and}  \quad  x_is_0.
\] 
\item \label{W-8} The elements $x_1s_1$ and $x_{-1}s_{-1}$.
\item \label{W-9} The elements 
\[
x_1s_{-1}+y_0(s_0+\scU^{n-1} \scV^{n-1}v), \quad \text{and} \quad x_{-1}s_1+y_0s_0.
\]
\item \label{W-10} The elements 
\[
x_{1}(s_0+\scU^{n-1} \scV^{n-1}v) \quad \text{and} \quad x_{-1}s_0.
\]
\end{enumerate}

As in the previous example, we are using the model of the involution
\[
(\iota_K \otimes \iota_K)\circ (\id \otimes \id+\Psi \otimes \Phi).
\]
\noindent Note in particular that we have
\begin{align*}
\iota_K(y_0v) &= y_0v+y_0u \\
\iota_K(y_0s_0) &= y_0s_0 + \scU^{n-1}\scV^{n-1}y_0 v \\
\iota_K (y_0u) &= y_0u \\
\iota_K (y_0s_1) &= y_0s_{-1}\\
\iota_K (y_0s_{-1}) &= y_0s_1 \\
\iota_K (x_{-1}v) &= (x_1(v+u) + y_0s_1)+y_0s_1 \\
\iota_K(x_1(v+u) + y_0s_1) &= x_{-1}v+ y_0s_{-1}.
\end{align*}

\begin{lem} \label{lem-staircase-box} The $\iota_K$-complex $\cS\otimes \cB_n$ decomposes as the direct sum of $\iota_K$-complexes $\cY \oplus \cW$.
\end{lem}

\begin{proof} Confirming that $\d$ and $\iota_K$ both preserve $\cY$ and $\cW$, and furthermore that $\cY\oplus \cW\iso \cS\otimes \cB_n$, proceeds straightforwardly and similarly to Lemma \ref{lem-cn-times-cn}.\end{proof}

\begin{proof}[Proof of Proposition \ref{prop-three-is-box}] We recall that by Lemma \ref{lem-cn-times-cn}, $\cCFK(2T_{2n, 2n+1})$ is $\iota_K$-locally equivalent to the complex $\cY_n$ of Figure \ref{fig:Yn}. Moreover, by Proposition \ref{prop:torusknots}, $\cCFK(T_{2n,4n+1})$ is $\iota_K$-locally equivalent to the staircase complex $\cD_n$ of Figure \ref{fig:Dn}. Applying Lemma \ref{lem-staircase-box} to $\cD_n \otimes \cB_n$ shows that $\cD_n\otimes \cB_n$ is $\iota_K$-locally equivalent to $\cY_n$. Therefore $\cY_n^{\vee}\otimes\cD_n$ is $\iota_K$-locally equivalent to $\cB_n^{\vee}$. The statement of the proposition follows immediately. \end{proof}

\subsection{The almost iota-complex associated to $S^3_{+1}(T_{2,3} \# -2T_{2n, 2n+1} \# T_{2n, 4n+1})$}

We now consider the tensor product of $\cB_n^\vee$ with the complex of the trefoil $T_{2,3}$, again for $n$ odd. Recall that $\cCFK(T_{2,3})$ is the staircase complex generated by three elements $r_0, s_1, s_{-1}$ with $\d(r_0) = \scV s_{-1} + \scU s_{1}$ and other differentials trivial. We are interested in the iota-complex $(E_n, \iota) = A_0(\cB_n^\vee \otimes \cCFK(T_{2,3}))$ obtained from the $\iota_K$-complex $\cB_n^\vee \otimes \cCFK(T_{2,3})$ by restricting to monomials $\scU^i \scV^j {\bf x}$ in $(\scU, \scV)^{-1}(\cB_n^\vee \otimes \cCFK(T_{2,3}))$ for which $A({\bf x})+j-i = 0$ and $i$ and $j$ are non-negative.

\begin{prop} \label{prop:almost-local-class} For $n\geq 3$ odd, the iota-complex $(E_n, \iota)$ is almost-locally equivalent to the standard complex $C(n-1)=C(+,-1,+,-n+1)$. \end{prop}

\begin{proof} The chain complex $(E_n, \iota) = A_0(\cB_n^\vee \otimes \cCFK(T_{2,3}))$ has fifteen generators and differentials as shown in Figure \ref{fig:A0}. (Recall that the action of $U$ is generated by the action of $\scU\scV$.)

\begin{figure}[ht]
\begin{tikzpicture}[scale=1.4,auto]
	\node (a) at (-.25,.75) {$a$};
	\node (b) at (-.25, -.25) {$b$};
	\node (c) at (.75,-.25) {$c$};
	\node (d) at (0,0) {$d$};
	\node (e) at (0,1) {$e$};
	\node (f) at (1,0) {$f$};
	\node (g) at (0,4) {$g$};
	\node (h) at (1,4) {$h$};
	\node (i) at (0,5) {$i$};
	\node (j) at (4,4) {$j$};
	\node (k) at (5,4) {$k$};
	\node (l) at (4,5) {$l$};
	\node (m) at (4,0) {$m$};
	\node (n) at (5,0) {$n$};
	\node (p) at (4,1) {$p$};
	\draw[->] (b) to (a);
	\draw[->] (b) to (c);
	\draw[->] (d) to (e);
	\draw[->] (d) to (f);
	\draw[->] (g) to (i);
	\draw[->] (j) to (l);
	\draw[->] (j) to (k);
	\draw[->] (m) to node {$\Ss{U}$} (p);
	\draw[->] (m) to (n);
	\draw[->] (g) to node [swap] {$\Ss{U}$} (h) ;
	\draw[->, bend right=20] (d) to (g);
	\draw[->, bend left=20] (e) to (i);
	\draw[->, bend left=20] (i) to node {$\Ss{U^n}$} (l);
	\draw[->, bend left=20] (g) to node {$\Ss{U^n}$} (j);
	\draw[->, bend right=20] (h) to node {$\Ss{U^{n-1}}$} (k);
	\draw[->, bend left=20] (e) to node {$\Ss{U}$} (p);
	\draw[->, bend right =20] (f) to node {$\Ss{U}$} (h);
	\draw[->, bend right=20] (f) to (n);
	\draw[->, bend left=20] (d) to (m);
	\draw[->, bend left=20] (p) to node {$\Ss{U^{n-1}}$} (l);
	\draw[->, bend right=20] (n) to node {$\Ss{U^{n}}$} (k);
	\draw[->, bend right=20] (m) to node {$\Ss{U^n}$} (j);
\end{tikzpicture}
\caption{The complex $A_0(\cB_n^\vee\otimes \cCFK(T_{2,3}))$.} \label{fig:A0}
\end{figure}
Using the usual model 
\[
\iota=(\iota_K \otimes \iota_K)\circ (\id \otimes \id+\Psi \otimes \Phi),
\]
 the involution takes the following form on $E_n$:
\[
\begin{split}
\iota(a)&=c+U^{n-1}k\\
\iota(b)&=b+U^{n-1}j\\
\iota(c)&=a+U^{n-1} l\\
\iota(d)&=d+b+p\\
\iota(e)&=c+ f\\
\iota(f)&=a+e\\
\iota(g)&=m\\
\iota(h)&=p\\
\end{split} \qquad \qquad \qquad \qquad
\begin{split}
\iota(i)&=n\\
\iota(j)&=j\\
\iota(k)&=l\\
\iota(l)&=k\\
\iota(m)&=g+U^{n-1} l\\
\iota(n)&=i\\
\iota(p)&=h\\
 & \\
\end{split}
\]
We now do a change of basis to $E_n$ to obtain the presentation of $E_n$ shown in Figure~\ref{fig:change-of-basis}.
\begin{figure}[ht]
\begin{tikzpicture}[scale=1.4,auto]
	\node (1-1) at (0,3) {$a$};
	\node (1-2) at (2,3) {$e+g$};
	\node (1-3) at (4,3) {$h+U^{n-1} j+p$};
	\node (1-4) at (6,3) {$p$};
	\node (1-5) at (8,3) {$l$};
	\node (2-1) at (0,2) {$b$};
	\node (2-2) at (2,2) {$a+c$};
	\node (2-3) at (4,2) {$j$};
	\node (2-4) at (6,2) {$l+k$};
	\node (3-1) at (0,1) {$e$};
	\node (3-2) at (2,1) {$i+Up$};
	\node (3-3) at (4,1) {$m$};
	\node (3-4) at (6,1) {$n+Up+U^n j$};
	\node (4-1) at (0,0) {$d$};
	\node (4-2) at (2,0) {$e+f+g+m$};
	\draw[->] (1-2) to node {$\Ss{U}$} (1-3);
	\draw[->] (1-4) to node {$\Ss{U^{n-1}}$} (1-5);
	\draw[->] (2-1) to (2-2);
	\draw[->] (2-3) to (2-4);
	\draw[->] (3-1) to (3-2);
	\draw[->] (3-3) to (3-4);
	\draw[->] (4-1) to (4-2);
\end{tikzpicture}
\caption{A new basis of $E_n$. Arrows denote the differential.}
\label{fig:change-of-basis}
\end{figure}

Let $F_n$ denote the top line of Figure~\ref{fig:change-of-basis}. There are projection and inclusion maps
\[
\Pi \colon E_n\to F_n \quad \text{and} \quad I\colon F_n\to E_n,
\]
which are obviously homotopy equivalences. In particular, $(F_n, \iota')$ is $\iota$-equivalent to $(E_n,\iota)$, where
\[
\iota'=\Pi \circ \iota \circ I.
\]
We compute
\begin{equation}
\begin{split}
\iota'(a)&=\Pi(c+U^{n-1}k)=a+U^{n-1} l\\
\iota'(e+g)&=\Pi(c+f+m)=a+e+g\\
\iota'(h+U^{n-1}j+p)&=\Pi(h+U^{n-1} j+p)=h+U^{n-1}j+p\\
\iota'(p)&=\Pi(h)=h+U^{n-1}j\\
\iota'(l)&=\Pi(k)=l.
\end{split}
\label{eq:iota'-def}
\end{equation}
We briefly remark how $\Pi$ is computed in~\eqref{eq:iota'-def}. The procedure is to write an element in terms of the basis in Figure~\ref{fig:change-of-basis}, and then project to the top row. As an example
\[
\Pi(c+U^{n-1} k)=\Pi(a+(a+c)+U^{n-1}l +U^{n-1} (l+k))=a+U^{n-1}l.
\]

We now consider the induced almost iota-complex. We claim that $(F_n,\iota')$ is $\iota$-homotopy equivalent equivalent to the complex $(F_n,\iota'')$ where $\iota''$  is the following map
\[
\begin{split}
\iota''(a)&=a\\
\iota''(e+g)&=a+(e+g)\\
\iota''(h+U^{n-1}j+p)&=h+U^{n-1}j+p\\
\iota''(p)&=(h+U^{n-1}j+ p)+p\\
\iota''(l)&=l.
\end{split}
\]
The equivalence of $(F_n,\iota')$ and $(F_n,\iota'')$ is seen as follows. The map $\iota'+\iota''$ sends $a$ to $U^{n-1} l$ and vanishes on all other generators of $F_n$. In particular, $\iota'+\iota''=[\d, H]$ on $F_n$, where $H$ is the $\bF[U]$-equivariant map which satisfies $H(a)=p$ and vanishes on all other generators.

However, $(F_n,\iota'')$ is the iota-complex
\[
\begin{tikzpicture}[scale=1.4,auto]
	\node (1-1) at (0,3) {$a$};
	\node (1-2) at (2,3) {$e+g$};
	\node (1-3) at (4,3) {$h+U^{n-1} j+p$};
	\node (1-4) at (6,3) {$p$};
	\node (1-5) at (8,3) {$l$};
	\draw[->] (1-2) to node {$\Ss{U}$} (1-3);
	\draw[->] (1-4) to node {$\Ss{U^{n-1}}$} (1-5);
	\draw[dashed,->] (1-2) to (1-1);
	\draw[dashed,->] (1-4) to (1-3);
\end{tikzpicture}
\]
where dashed arrows denote $\omega:=\iota''+\id$. This clearly reduces to the almost iota-complex $C(+,-1,+,-n+1)=C(n-1)$. \end{proof}

\section{Tensor products of almost iota-complexes} \label{section-4-tensorproducts}

\subsection{The subgroup of the group of almost iota-complexes spanned by $C(n)$}

We now compute the subgroup of the group of almost iota-complexes spanned by linear combinations of the almost iota-complexes $C(n)=(+, -1, +, -n)$ for varying $n>1$. The results of this section are similar to \cite[Section 8.1]{DHSThomcob}. In this section we use the $+$ symbol instead of $\otimes$ to represent the tensor product of almost iota complexes. Observe that $-C(n)$ is parametrized by $(-, 1, -, n)$. We will consider sums of the form
\[ C =  \pm C(n_1) \pm C(n_2) \pm \dots \pm C(n_m), \]
where each $n_k > 0$. Without loss of generality, we assume that the $n_k$ are non-increasing, that is, $n_1 \geq n_2 \geq \dots \geq n_m$. Furthermore, we assume that $C$ is \emph{fully simplified}, meaning that if $n_i = n_{i+1}$, the complexes $C(n_i)$ and $C(n_{i+1})$ occur with the same sign. The following theorem and its proof are analogous to \cite[Theorem 8.1]{DHSThomcob}.

\begin{thm}\label{thm:sumCi}
Let
\[ C =  \pm C(n_1) \pm C(n_2) \pm \dots \pm C(n_m) \]
be fully simplified with $n_1 \geq n_2 \geq \dots \geq n_m>1$. Then the standard representative of $C$ is obtained by concatenating the parameters of the above terms in the order that they appear.
\end{thm}

\begin{example}
The standard representative of $C(n_1) + C(n_2) + \dots + C(n_m)$ is 
\[(+, -1, +, -n_1) + \dots  + ( +, -1, +, -n_m) = (+, -1, +, -n_1, +, -1, +, -n_2, \dots, +, -1, +, -n_m). \]
\end{example}

\begin{example}
The standard representative of $ -C(n_1) - C(n_2) - \dots - C(n_m)$ is 
\[(-, 1, -, n_1)+ \dots + (-, 1, -, n_m) = (-, 1, -, n_1, -, 1, -, n_2, \dots, -, 1, -, n_m). \]
\end{example}

\begin{example}
The standard representative of $C(n_1) - C(n_2)$ is 
\[ (+, -1, +, -n_1) + (-, 1, -, n_2) = (+, -1, +, -n_1, -, 1, -, n_2). \]
\end{example}

\begin{proof}[Proof of Theorem \ref{thm:sumCi}]
This proof closely follows the proof of \cite[Theorem 8.1]{DHSThomcob}. We begin with a model calculation in the case $m=2$. Let $N$ and $M$ be positive integers and consider $C(N) = (+, -1, +, -N)$ and $C(M)=(+, -1, +, -M)$. We consider the following two cases:
\begin{enumerate}
	\item $C_1 = -C(N) - C(M)$ with $N \geq M$,
	\item $C_2= C(N) - C(M)$ with $N>M$,
\end{enumerate}
and show that we have the following almost local equivalences
\[ C_1 \sim (-, 1, -, N, -, 1, -, M) \qquad \textup{ and } \qquad C_2 \sim (+, -1, +, -N, -, 1, -, M) . \]

The other two cases $-C(N) + C(M)$ and $C(N) + C(M)$ follow by dualizing. For both $C_1$ and $C_2$, the obvious tensor product basis consists of 25 generators. These bases are displayed in the left of Figures \ref{fig:C1} and \ref{fig:C2}, where they are labeled $a$ through $y$. The dashed red arrows represent the action of $\omega$ and the solid black arrows represent $\d$, with the label over the arrow denoting the associated power of $U$; for example, in $C_1$, we have that $\d o = U j + U^M n $ and that $\omega (m) = n+r+s$.

On the right of Figure \ref{fig:C1}, we have performed the change of basis 
\begin{align*}
	f' &= f+b+g \\
	k' &= k+c+\ell \\
	p' &= p+d+i+h+\ell \\
	u' &= u+U^{N-M}e+U^{N-1}m+U^{N-1}n \\
	v' &= v+U^{N-M}j + U^{N-1}n+U^{N-1}r \\
	w' &= w+U^{N-M}o \\
	x' &= x+ U^{N-M} t,
\end{align*}
keeping the other basis elements the same. The reader should verify that this results in the diagram in the right of Figure \ref{fig:C1}. It is then evident from the right of Figure \ref{fig:C1} that $C(N)+C(M)$ is almost locally equivalent to
\[  (-, 1, -, N, -, 1, -, M), \]
as desired.

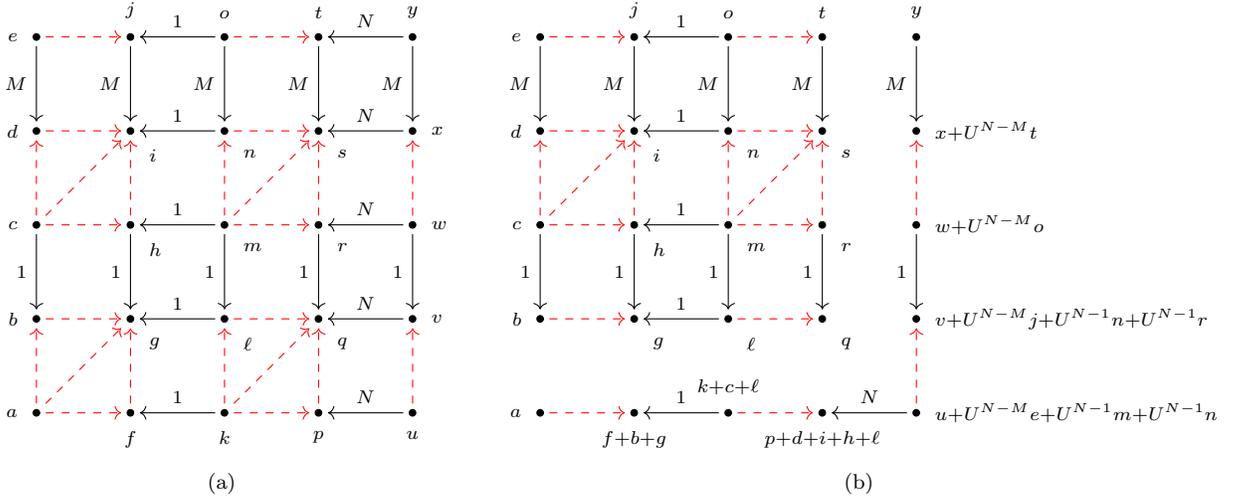
\begin{figure}[ht!]
\[
\subfigure[]{
\begin{tikzpicture}[scale=1.25]

	\filldraw (0, 4) circle (1pt) node[label= left:{\lab{e}}] (e) {};
	\filldraw (0, 3) circle (1pt) node[label= left:{\lab{d}}] (d) {};
	\filldraw (0, 2) circle (1pt) node[label= left : {\lab{c}}] (c) {};
	\filldraw (0, 1) circle (1pt) node[label= left : {\lab{b}}] (b) {};
	\filldraw (0, 0) circle (1pt) node[label= left : {\lab{a}}] (a) {};
	
	\filldraw (1, 4) circle (1pt) node[label= above:{\lab{j}}] (j) {};
	\filldraw (1, 3) circle (1pt) node[label=below right:{\lab{i}}] (i) {};
	\filldraw (1, 2) circle (1pt) node[label=below right : {\lab{h}}] (h) {};
	\filldraw (1, 1) circle (1pt) node[label=below right : {\lab{g}}] (g) {};
	\filldraw (1, 0) circle (1pt) node[label= below : {\lab{f}}] (f) {};

	\filldraw (2, 4) circle (1pt) node[label= above:{\lab{o}}] (o) {};
	\filldraw (2, 3) circle (1pt) node[label=below right:{\lab{n}}] (n) {};
	\filldraw (2, 2) circle (1pt) node[label=below right : {\lab{m}}] (m) {};
	\filldraw (2, 1) circle (1pt) node[label=below right : {\lab{\ell}}] (l) {};
	\filldraw (2, 0) circle (1pt) node[label= below : {\lab{k}}] (k) {};

	\filldraw (3, 4) circle (1pt) node[label= above:{\lab{t}}] (t) {};
	\filldraw (3, 3) circle (1pt) node[label=below right:{\lab{s}}] (s) {};
	\filldraw (3, 2) circle (1pt) node[label=below right : {\lab{r}}] (r) {};
	\filldraw (3, 1) circle (1pt) node[label=below right : {\lab{q}}] (q) {};
	\filldraw (3, 0) circle (1pt) node[label= below : {\lab{p}}] (p) {};	

	\filldraw (4, 4) circle (1pt) node[label= above:{\lab{y}}] (y) {};
	\filldraw (4, 3) circle (1pt) node[label= right:{\lab{x}}] (x) {};
	\filldraw (4, 2) circle (1pt) node[label= right : {\lab{w}}] (w) {};
	\filldraw (4, 1) circle (1pt) node[label= right : {\lab{v}}] (v) {};
	\filldraw (4, 0) circle (1pt) node[label= below : {\lab{u}}] (u) {};	
	
	\draw[->, red, dashed] (a) to (f);
	\draw[->] (k) to node[above]{\lab{1}} (f);	
	\draw[->, red, dashed] (k) to (p);	
	\draw[->] (u) to node[above]{\lab{N}}  (p);
	
	\draw[->, red, dashed] (a) to (b);
	\draw[->, red, dashed] (a) to (g);
	\draw[->, red, dashed] (f) to (g);
	\draw[->, red, dashed] (k) to (l);
	\draw[->, red, dashed] (k) to (q);
	\draw[->, red, dashed] (p) to (q);
	\draw[->, red, dashed] (u) to (v);		

	\draw[->, red, dashed] (b) to (g);
	\draw[->] (l) to node[above]{\lab{1}} (g);	
	\draw[->, red, dashed] (l) to (q);	
	\draw[->] (v) to node[above]{\lab{N}}  (q);
	
	\draw[->] (c) to node[left]{\lab{1}} (b);	
	\draw[->] (h) to node[left]{\lab{1}} (g);	
	\draw[->] (m) to node[left]{\lab{1}} (l);	
	\draw[->] (r) to node[left]{\lab{1}} (q);	
	\draw[->] (w) to node[left]{\lab{1}} (v);	

	\draw[->, red, dashed] (c) to (h);
	\draw[->] (m) to node[above]{\lab{1}} (h);	
	\draw[->, red, dashed] (m) to (r);	
	\draw[->] (w) to node[above]{\lab{N}}  (r);
	
	\draw[->, red, dashed] (c) to (d);
	\draw[->, red, dashed] (c) to (i);
	\draw[->, red, dashed] (h) to (i);
	\draw[->, red, dashed] (m) to (n);
	\draw[->, red, dashed] (m) to (s);
	\draw[->, red, dashed] (r) to (s);
	\draw[->, red, dashed] (w) to (x);	

	\draw[->, red, dashed] (d) to (i);
	\draw[->] (n) to node[above]{\lab{1}} (i);	
	\draw[->, red, dashed] (n) to (s);	
	\draw[->] (x) to node[above]{\lab{N}}  (s);	

	\draw[->] (e) to node[left]{\lab{M}} (d);	
	\draw[->] (j) to node[left]{\lab{M}} (i);	
	\draw[->] (o) to node[left]{\lab{M}} (n);	
	\draw[->] (t) to node[left]{\lab{M}} (s);	
	\draw[->] (y) to node[left]{\lab{M}} (x);	

	\draw[->, red, dashed] (e) to (j);
	\draw[->] (o) to node[above]{\lab{1}} (j);	
	\draw[->, red, dashed] (o) to (t);	
	\draw[->] (y) to node[above]{\lab{N}}  (t);	

\end{tikzpicture}
}
\hspace{10pt}
\subfigure[]{
\begin{tikzpicture}[scale=1.25]

	\filldraw (0, 4) circle (1pt) node[label= left:{\lab{e}}] (e) {};
	\filldraw (0, 3) circle (1pt) node[label= left:{\lab{d}}] (d) {};
	\filldraw (0, 2) circle (1pt) node[label= left : {\lab{c}}] (c) {};
	\filldraw (0, 1) circle (1pt) node[label= left : {\lab{b}}] (b) {};
	\filldraw (0, 0) circle (1pt) node[label= left : {\lab{a}}] (a) {};
	
	\filldraw (1, 4) circle (1pt) node[label= above:{\lab{j}}] (j) {};
	\filldraw (1, 3) circle (1pt) node[label=below right:{\lab{i}}] (i) {};
	\filldraw (1, 2) circle (1pt) node[label=below right : {\lab{h}}] (h) {};
	\filldraw (1, 1) circle (1pt) node[label=below right : {\lab{g}}] (g) {};
	\filldraw (1, 0) circle (1pt) node[label= below : {\lab{f+b+g}}] (f) {};

	\filldraw (2, 4) circle (1pt) node[label= above:{\lab{o}}] (o) {};
	\filldraw (2, 3) circle (1pt) node[label=below right:{\lab{n}}] (n) {};
	\filldraw (2, 2) circle (1pt) node[label=below right : {\lab{m}}] (m) {};
	\filldraw (2, 1) circle (1pt) node[label=below right : {\lab{\ell}}] (l) {};
	\filldraw (2, 0) circle (1pt) node[label= above : {\lab{k+c+\ell}}] (k) {};

	\filldraw (3, 4) circle (1pt) node[label= above:{\lab{t}}] (t) {};
	\filldraw (3, 3) circle (1pt) node[label=below right:{\lab{s}}] (s) {};
	\filldraw (3, 2) circle (1pt) node[label=below right : {\lab{r}}] (r) {};
	\filldraw (3, 1) circle (1pt) node[label=below right : {\lab{q}}] (q) {};
	\filldraw (3, 0) circle (1pt) node[label= below : {\lab{p+d+i+h+\ell}}] (p) {};	

	\filldraw (4, 4) circle (1pt) node[label= above:{\lab{y}}] (y) {};
	\filldraw (4, 3) circle (1pt) node[label= right:{\lab{x+U^{N-M}t}}] (x) {};
	\filldraw (4, 2) circle (1pt) node[label= right : {\lab{w+U^{N-M}o}}] (w) {};
	\filldraw (4, 1) circle (1pt) node[label= right : {\lab{v+U^{N-M}j + U^{N-1}n+U^{N-1}r}}] (v) {};
	\filldraw (4, 0) circle (1pt) node[label= right : {\lab{u+U^{N-M}e+U^{N-1}m+U^{N-1}n}}] (u) {};	
	
	\draw[->, red, dashed] (a) to (f);
	\draw[->] (k) to node[above]{\lab{1}} (f);	
	\draw[->, red, dashed] (k) to (p);	
	\draw[->] (u) to node[above]{\lab{N}}  (p);
	
	\draw[->, red, dashed] (u) to (v);		

	\draw[->, red, dashed] (b) to (g);
	\draw[->] (l) to node[above]{\lab{1}} (g);	
	\draw[->, red, dashed] (l) to (q);

	\draw[->] (c) to node[left]{\lab{1}} (b);	
	\draw[->] (h) to node[left]{\lab{1}} (g);	
	\draw[->] (m) to node[left]{\lab{1}} (l);	
	\draw[->] (r) to node[left]{\lab{1}} (q);	
	\draw[->] (w) to node[left]{\lab{1}} (v);	

	\draw[->, red, dashed] (c) to (h);
	\draw[->] (m) to node[above]{\lab{1}} (h);	
	\draw[->, red, dashed] (m) to (r);

	\draw[->, red, dashed] (c) to (d);
	\draw[->, red, dashed] (c) to (i);
	\draw[->, red, dashed] (h) to (i);
	\draw[->, red, dashed] (m) to (n);
	\draw[->, red, dashed] (m) to (s);
	\draw[->, red, dashed] (r) to (s);
	\draw[->, red, dashed] (w) to (x);	

	\draw[->, red, dashed] (d) to (i);
	\draw[->] (n) to node[above]{\lab{1}} (i);	
	\draw[->, red, dashed] (n) to (s);

	\draw[->] (e) to node[left]{\lab{M}} (d);	
	\draw[->] (j) to node[left]{\lab{M}} (i);	
	\draw[->] (o) to node[left]{\lab{M}} (n);	
	\draw[->] (t) to node[left]{\lab{M}} (s);	
	\draw[->] (y) to node[left]{\lab{M}} (x);	

	\draw[->, red, dashed] (e) to (j);
	\draw[->] (o) to node[above]{\lab{1}} (j);	
	\draw[->, red, dashed] (o) to (t);

\end{tikzpicture}
}
\]
\caption{Left, the obvious tensor product basis for $C_1$. Right, after a change of basis. Recall that $N \geq M$.}
\label{fig:C1}
\end{figure}

The computation of $C(N) - C(M)$ is similar.  On the right of Figure \ref{fig:C2}, we have performed a change of basis
\begin{align*}
	a' &= a+b+g+m+n+s+U^{N-M}y \\
	f' &= f+r \\
	k' &= k+q+U^{N-1}w \\
	p' &= p+k \\
	q' &= q+U^{N-1}w \\
	s' &= s+U^{N-M} y,
\end{align*}
keeping the other basis elements the same (e.g., $b'=b$, etc). The reader should verify that this results in the diagram on the right of Figure \ref{fig:C2}, where we consider $\omega$ modulo $U$. For example, 
\begin{align*}
	\omega(f+r) &= a+b+g+m+n+s \\
		&\equiv a+b+g+m+n+s+U^{N-M}y \mod U.
\end{align*}
We have marked the dashed red arrows that are congruence modulo $U$ (rather than equality) with congruence symbols to emphasize this point. (Here is where we first use the notion of almost local equivalence; in the computations of Section \ref{section-3-iotacomplex}, all of the maps were local equivalences.) Note that since $N > M$, we have that $N-M > 0$. It is then evident from the right of Figure \ref{fig:C2} that $C(N) - C(M)$ is almost locally equivalent to the standard complex $(+, -1, +, -N, -, 1, -, M)$, as desired.

\begin{figure}[ht!]
\[
\subfigure[]{
\begin{tikzpicture}[scale=1.25]

	\filldraw (0, 4) circle (1pt) node[label= left:{\lab{e}}] (e) {};
	\filldraw (0, 3) circle (1pt) node[label= left:{\lab{d}}] (d) {};
	\filldraw (0, 2) circle (1pt) node[label= left : {\lab{c}}] (c) {};
	\filldraw (0, 1) circle (1pt) node[label= left : {\lab{b}}] (b) {};
	\filldraw (0, 0) circle (1pt) node[label= left : {\lab{a}}] (a) {};
	
	\filldraw (1, 4) circle (1pt) node[label= above:{\lab{j}}] (j) {};
	\filldraw (1, 3) circle (1pt) node[label=below left:{\lab{i}}] (i) {};
	\filldraw (1, 2) circle (1pt) node[label=below left : {\lab{h}}] (h) {};
	\filldraw (1, 1) circle (1pt) node[label=below left : {\lab{g}}] (g) {};
	\filldraw (1, 0) circle (1pt) node[label= below : {\lab{f}}] (f) {};

	\filldraw (2, 4) circle (1pt) node[label= above:{\lab{o}}] (o) {};
	\filldraw (2, 3) circle (1pt) node[label=below left:{\lab{n}}] (n) {};
	\filldraw (2, 2) circle (1pt) node[label=below left : {\lab{m}}] (m) {};
	\filldraw (2, 1) circle (1pt) node[label=below left : {\lab{\ell}}] (l) {};
	\filldraw (2, 0) circle (1pt) node[label= below : {\lab{k}}] (k) {};

	\filldraw (3, 4) circle (1pt) node[label= above:{\lab{t}}] (t) {};
	\filldraw (3, 3) circle (1pt) node[label=below left:{\lab{s}}] (s) {};
	\filldraw (3, 2) circle (1pt) node[label=below left : {\lab{r}}] (r) {};
	\filldraw (3, 1) circle (1pt) node[label=below left : {\lab{q}}] (q) {};
	\filldraw (3, 0) circle (1pt) node[label= below : {\lab{p}}] (p) {};	

	\filldraw (4, 4) circle (1pt) node[label= above:{\lab{y}}] (y) {};
	\filldraw (4, 3) circle (1pt) node[label= right:{\lab{x}}] (x) {};
	\filldraw (4, 2) circle (1pt) node[label= right : {\lab{w}}] (w) {};
	\filldraw (4, 1) circle (1pt) node[label= right : {\lab{v}}] (v) {};
	\filldraw (4, 0) circle (1pt) node[label= below : {\lab{u}}] (u) {};	
	
	\draw[->, red, dashed] (f) to (a);
	\draw[->] (f) to node[above]{\lab{1}} (k);	
	\draw[->, red, dashed] (p) to (k);	
	\draw[->] (p) to node[above]{\lab{N}}  (u);
	
	\draw[->, red, dashed] (a) to (b);
	\draw[->, red, dashed] (f) to (b);
	\draw[->, red, dashed] (f) to (g);
	\draw[->, red, dashed] (k) to (l);
	\draw[->, red, dashed] (p) to (l);
	\draw[->, red, dashed] (p) to (q);
	\draw[->, red, dashed] (u) to (v);		

	\draw[->, red, dashed] (g) to (b);
	\draw[->] (g) to node[above]{\lab{1}} (l);	
	\draw[->, red, dashed] (q) to (l);	
	\draw[->] (q) to node[above]{\lab{N}}  (v);
	
	\draw[->] (c) to node[left]{\lab{1}} (b);	
	\draw[->] (h) to node[left]{\lab{1}} (g);	
	\draw[->] (m) to node[left]{\lab{1}} (l);	
	\draw[->] (r) to node[left]{\lab{1}} (q);	
	\draw[->] (w) to node[left]{\lab{1}} (v);	

	\draw[->, red, dashed] (h) to (c);
	\draw[->] (h) to node[above]{\lab{1}} (m);	
	\draw[->, red, dashed] (r) to (m);	
	\draw[->] (r) to node[above]{\lab{N}}  (w);
	
	\draw[->, red, dashed] (c) to (d);
	\draw[->, red, dashed] (h) to (d);
	\draw[->, red, dashed] (h) to (i);
	\draw[->, red, dashed] (m) to (n);
	\draw[->, red, dashed] (r) to (n);
	\draw[->, red, dashed] (r) to (s);
	\draw[->, red, dashed] (w) to (x);	

	\draw[->, red, dashed] (i) to (d);
	\draw[->] (i) to node[above]{\lab{1}} (n);	
	\draw[->, red, dashed] (s) to (n);	
	\draw[->] (s) to node[above]{\lab{N}}  (x);	

	\draw[->] (e) to node[left]{\lab{M}} (d);	
	\draw[->] (j) to node[left]{\lab{M}} (i);	
	\draw[->] (o) to node[left]{\lab{M}} (n);	
	\draw[->] (t) to node[left]{\lab{M}} (s);	
	\draw[->] (y) to node[left]{\lab{M}} (x);	

	\draw[->, red, dashed] (j) to (e);
	\draw[->] (j) to node[above]{\lab{1}} (o);	
	\draw[->, red, dashed] (t) to (o);	
	\draw[->] (t) to node[above]{\lab{N}}  (y);	

\end{tikzpicture}
}
\hspace{10pt}
\subfigure[]{
\begin{tikzpicture}[scale=1.25]

	\filldraw (0, 4) circle (1pt) node[label= left:{\lab{e}}] (e) {};
	\filldraw (0, 3) circle (1pt) node[label= left:{\lab{d}}] (d) {};
	\filldraw (0, 2) circle (1pt) node[label= left : {\lab{c}}] (c) {};
	\filldraw (0, 1) circle (1pt) node[label= left : {\lab{b}}] (b) {};
	\filldraw (0, 0) circle (1pt) node[label= {[xshift=-15pt, yshift=-20pt]{\lab{a+b+g+m+n+s+U^{N-M}y}}}] (a) {};
	
	\filldraw (1, 4) circle (1pt) node[label= above:{\lab{j}}] (j) {};
	\filldraw (1, 3) circle (1pt) node[label=below left:{\lab{i}}] (i) {};
	\filldraw (1, 2) circle (1pt) node[label=below left : {\lab{h}}] (h) {};
	\filldraw (1, 1) circle (1pt) node[label=below left : {\lab{g}}] (g) {};
	\filldraw (1, 0) circle (1pt) node[label= above : {\lab{f+r}}] (f) {};

	\filldraw (2, 4) circle (1pt) node[label= above:{\lab{o}}] (o) {};
	\filldraw (2, 3) circle (1pt) node[label=below left:{\lab{n}}] (n) {};
	\filldraw (2, 2) circle (1pt) node[label=below left : {\lab{m}}] (m) {};
	\filldraw (2, 1) circle (1pt) node[label=below left : {\lab{\ell}}] (l) {};
	\filldraw (2, 0) circle (1pt) node[label= below : {\lab{k+q+U^{N-1}w}}] (k) {};

	\filldraw (3, 4) circle (1pt) node[label= above:{\lab{t}}] (t) {};
	\filldraw (3, 3) circle (1pt) node[label=below :{\lab{s+U^{N-M}y}}] (s) {};
	\filldraw (3, 2) circle (1pt) node[label=below left : {\lab{r}}] (r) {};
	\filldraw (3, 1) circle (1pt) node[label=below : {\lab{q+U^{N-1}w}}] (q) {};
	\filldraw (3, 0) circle (1pt) node[label= above : {\lab{p+k}}] (p) {};	

	\filldraw (4, 4) circle (1pt) node[label= above:{\lab{y}}] (y) {};
	\filldraw (4, 3) circle (1pt) node[label= right:{\lab{x}}] (x) {};
	\filldraw (4, 2) circle (1pt) node[label= right : {\lab{w}}] (w) {};
	\filldraw (4, 1) circle (1pt) node[label= right : {\lab{v}}] (v) {};
	\filldraw (4, 0) circle (1pt) node[label= below : {\lab{u}}] (u) {};	
	
	\draw[->, red, dashed] (f) to  node[above]{\lab{\equiv}} (a);
	\draw[->] (f) to node[above]{\lab{1}} (k);	
	\draw[->, red, dashed] (p) to node[above]{\lab{\equiv}}  (k);	
	\draw[->] (p) to node[above]{\lab{N}}  (u);

	\draw[->, red, dashed] (u) to (v);		

	\draw[->, red, dashed] (g) to (b);
	\draw[->] (g) to node[above]{\lab{1}} (l);	
	\draw[->, red, dashed] (q) to node[above]{\lab{\equiv}} (l);

	\draw[->] (c) to node[left]{\lab{1}} (b);	
	\draw[->] (h) to node[left]{\lab{1}} (g);	
	\draw[->] (m) to node[left]{\lab{1}} (l);	
	\draw[->] (r) to node[left]{\lab{1}} (q);	
	\draw[->] (w) to node[left]{\lab{1}} (v);	

	\draw[->, red, dashed] (h) to (c);
	\draw[->] (h) to node[above]{\lab{1}} (m);	
	\draw[->, red, dashed] (r) to (m);

	\draw[->, red, dashed] (c) to (d);
	\draw[->, red, dashed] (h) to (d);
	\draw[->, red, dashed] (h) to (i);
	\draw[->, red, dashed] (m) to (n);
	\draw[->, red, dashed] (r) to (n);
	\draw[->, red, dashed] (r) to node[right]{\lab{\equiv}}  (s);
	\draw[->, red, dashed] (w) to (x);	

	\draw[->, red, dashed] (i) to (d);
	\draw[->] (i) to node[above]{\lab{1}} (n);	
	\draw[->, red, dashed] (s) to (n);

	\draw[->] (e) to node[left]{\lab{M}} (d);	
	\draw[->] (j) to node[left]{\lab{M}} (i);	
	\draw[->] (o) to node[left]{\lab{M}} (n);	
	\draw[->] (t) to node[left]{\lab{M}} (s);	
	\draw[->] (y) to node[left]{\lab{M}} (x);	

	\draw[->, red, dashed] (j) to (e);
	\draw[->] (j) to node[above]{\lab{1}} (o);	
	\draw[->, red, dashed] (t) to (o);

\end{tikzpicture}
}
\]
\caption{Left, the obvious tensor product basis for $C_2$. Right, after a change of basis. Recall that $N > M$.}
\label{fig:C2}
\end{figure}

We now consider the general case, by induction on $m$. Suppose we have established the claim for 
\[ C= \pm C(n_1) \pm C(n_2) \pm \dots \pm C(n_m) \]
as in the statement of the theorem. Let $M$ be a positive integer such that $M \leq n_m$. Now consider
\[ C' = C - C(M). \]
The case $C+C(M)$ where we add rather than subtract $C(M)$ follows by dualizing. The obvious tensor product basis for $C-C(M)$ is schematically depicted in Figure \ref{fig:bigbasischange} (where we have arbitrarily chosen signs in front of each $C(n_k)$). Using the inductive hypothesis applied to $C$, this complex has $5(4m+1)$ generators.

Our strategy will be to split off subcomplexes by change-of-basis moves paralleling those defined for $C_1$ and $C_2$. We begin by comparing the leftmost 25 generators of $C'$. Label these $a$ through $y$, as usual. We begin by letting $n_1$ assume the role of $N$ from the previous argument, so that applying the appropriate change of basis based as in Figure \ref{fig:C1} if the coefficient of $C(n_1)$ is negative and as in Figure \ref{fig:C2} if the coefficient of $C(n_1)$ is positive results in the second row of Figure \ref{fig:bigbasischange}. Note that in the first case, there is an additional subtlety: since we replace $u, v, w, x$ with $u', v', w', x'$ respectively, we are in danger of changing the dashed red arrows entering/exiting $u, v, w$, and $x$ on the right. To check that this does not happen, we consider two cases:
\begin{enumerate}
	\item Suppose that there are dashed red arrows entering $u, v, w, x$ from the right. We claim that in order for this to happen, we must have $n_1 > M$. Indeed, because $C$ is fully simplified, if $n_1=M$, then all subsequent terms in our sum are $-C(M)$, in which case $u, v, w, x$ would have dashed red arrows exiting them, rather than entering. Hence $n_1 > M$. But this shows that 
\begin{align*}
	u' &\equiv u \mod U \\
	v' &\equiv v \mod U \\
	w' &\equiv w \mod U \\
	x' &\equiv x \mod U,	
\end{align*}
which means that the original dashed red arrows hold modulo $U$.
	\item Suppose that there are dashed red arrows exiting $u, v, w, x$ to the right. Then we can explicitly check that the dashed red arrows exiting $u', v', w', x'$ are unchanged:
\begin{align*}
	\omega (u') &= 
		\omega (u) + U^{n_1-M}j + U^{n_1-1} (n+r) = \omega (u) + v' + v  \\
	\omega (v') &=  
		\omega(v)  \\
	\omega (w') &= 
		\omega(w) + U^{n_1-M} t = \omega(w) + x' + x \\
	\omega (x') &= (x + U^{n_1-M} t) = \omega (x).
\end{align*}
\end{enumerate}
In particular, we see that in either case, our change of basis does not change the form of the diagram lying to the right of $u, v, w, x$, and $y$.

We now consider the 25 generators lying inside the dashed box in the second row of Figure \ref{fig:bigbasischange}, relabeling them $a$ through $y$ as usual. Again, we attempt to perform a change of basis as in Figure \ref{fig:C1} or \ref{fig:C2},  now with $n_2$ taking the role of $N$ from the initial argument, as follows. 

\begin{figure}[ht!]
\[
\subfigure[]{

}
\]

\caption{}
\label{fig:bigbasischange}
\end{figure}

\begin{figure}[ht!]
\[
\subfigure[]{
%
}
\]

\caption{}
\label{fig:retro}
\end{figure}

If the second term $C(n_2)$ appears with negative sign in $C$, then we use the change of basis in Figure \ref{fig:C1}. 

If the second term $C(n_2)$ appears with positive sign in $C$, then we attempt to use the change of basis in Figure \ref{fig:C2}. However, there is an additional subtlety, as depicted in Figure \ref{fig:retro}. Namely, we have a black arrow entering/exiting $a$ from the left, so when we set 
\[ a' = a+b+g+m+n+s+U^{n_1-M}y, \]
we must ensure that we don't change the diagram to the left of the dashed box. If the black arrow to the left of $a$ is exiting $a$, then this follows from the fact that $\d a' = \d a$. However, if the arrow is instead entering $a$ (representing the relation $\d \alpha = U^{n_1} a$), then the diagram is no longer accurate, since evidently $\d \alpha \neq U^{n_1} a'$. In this situation, we carry out the additional (retroactive) basis change
\[ \alpha' = \alpha + U^{n_1-1} (c+h+i) + U^{n_1-M} t\]  
as in Figure \ref{fig:retro}, so that $\d \alpha = U^\zeta a'$. 
Note that $n_1 > 1$ by hypothesis. Furthermore, note that $n_1>M$, since $C$ is fully simplified. Hence $\alpha' \equiv \alpha \mod U$, so modulo $U$, our basis change does not change the dashed red arrow leaving $\alpha$. In any case, we see that performing the appropriate change-of-basis splits off another subcomplex and leads to a diagram as in the third row of Figure \ref{fig:bigbasischange}. Iterating this procedure results in the complex depicted in the bottom row of Figure \ref{fig:bigbasischange}, as desired.
\end{proof}

\subsection{Proof of Theorem \ref{thm:main}} We are now ready to complete the proof of our main theorem.

\begin{proof}[Proof of Theorem \ref{thm:main}] By Proposition \ref{prop:surgery-formula}, the iota-complex $\CF^-(S_{+1}(T_{2,3}\#-2T_{2n,2n+1} \# T_{2n,4n+1} ), \iota)$ is locally equivalent to $(A_0(T_{2,3}\#-2T_{2n,2n+1} \# T_{2n,4n+1}), \iota_K)$. By Proposition \ref{prop-three-is-box} and Proposition \ref{prop:almost-local-class}, for $n \geq 3$ odd, the almost local equivalence class of $(A_0(T_{2,3}\# -2T_{2n,2n+1} \# T_{2n,4n+1}), \iota_K)$ is $C(n-1)$. Theorem \ref{thm:sumCi} implies that the complexes $C(n)$ span a $\Z^{\infty}$ subgroup in $\widehat{\frI}$; in particular, elements in this subgroup of $\widehat{\frI}$ are of the form
\[ (a_1, b_1, \dots, a_{2m}, b_{2m}), \]
where 
\[ |b_1| = |b_3| = \ldots = |b_{2m-1}| = 1 \]
and
\[ |b_2| \geq |b_4| \geq \dots \geq |b_{2m}| \]
By \cite[Theorem 8.1]{DHSThomcob}, elements in $\hat{h}(\Theta^3_{\SF})$ are of the form
\[ (a_1, b_1, \dots, a_m, b_m), \]
where 
\[ |b_1| \geq |b_2| \geq \dots \geq |b_m| \]
and so the span of the $C(n)$ intersects $\hat{h}(\Theta^3_{\SF})$ trivially. Therefore, we conclude that the classes 
\[[S_{+1}(T_{2,3}\#-2T_{2n,2n+1} \# T_{2n,4n+1})] \]
span a $\Z^{\infty}$ subgroup of $\Theta^3_{\mathbb Z}/\Theta^3_{\SF}$.

\end{proof}

\clearpage

\bibliographystyle{custom}
\def\MR#1{}
\bibliography{biblio}

\end{document}